  \RequirePackage{fix-cm}

  \documentclass[smallextended,envcountsect]{svjour3}       
  \smartqed  
  \usepackage{graphicx}
  \usepackage{braket,amsfonts}
  \usepackage{array}
  \usepackage{pgfplots}
  \usepackage{multicol}
  \usepackage{diagbox}
  \usepackage{amssymb}
  \usepackage{amsmath,bm}
  \usepackage{multirow}
  \usepackage{makecell}
  \usepackage{graphicx}
  \usepackage{subfigure}
  \usepackage{float}
  \usepackage[titletoc, title]{appendix}
  \usepackage{lipsum}
  \usepackage[left=2.5cm,right=2.5cm,top=2.5cm, bottom=2.5cm]{geometry}
  \usepackage{amsfonts}
  \usepackage{epstopdf}
  \usepackage{algorithmic}
  \usepackage{amsopn}
  \usepackage{mathrsfs}
  \usepackage{algorithmic}
  \usepackage{graphicx,epstopdf}
  \usepackage{amsopn}
  \usepackage{hyperref}
  \usepackage{xspace}
  \usepackage{xcolor}
  \usepackage{bold-extra}
  \usepackage{cleveref}
  \newcommand{\zzy}[1]{\textcolor{black}{#1}}
  \newcommand{\zzya}[1]{\textcolor{black}{#1}}
  \newcommand{\zzyb}[1]{\textcolor{black}{#1}}
  \newcommand{\zzyc}[1]{\textcolor{red}{#1}}
\begin{document}

\title{\zzya{Supercloseness} error estimates for the div least-squares finite element method on elliptic problems}

\titlerunning{\zzya{Supercloseness} error estimates for the least-squares method}        

\author{Gang Chen         \and
  Fanyi Yang\and
  Zheyuan Zhang 
}


\institute{Gang Chen         \and
  Fanyi Yang\and
  Zheyuan Zhang \at
  School of Mathematics, Sichuan University, Chengdu, China. \\
  \email{cglwdm@scu.edu.cn}\and\email{yangfanyi@scu.edu.cn}\and\email{zhangzheyuan@stu.scu.edu.cn}
}

\date{Received: date / Accepted: date}

\maketitle

\begin{abstract}
  In this paper we \zzy{provide some error estimates }for the div least-squares finite element method on elliptic problems. The main contribution is presenting a complete error analysis, which improves the current \emph{state-of-the-art} results. The error estimates for both the scalar and the flux variables are established by \zzy{specially} designed dual arguments with \zzy{the help of two projections: elliptic projection and H(div) projection, which are crucial to \zzya{supercloseness} estimates.} \zzy{In most cases, $H^3$ regularity is omitted to get the optimal convergence rate for vector and scalar unknowns, and most of our results require a lower regularity for the vector variable than the scalar.} Moreover, a series of \zzya{supercloseness} results are proved, which are \emph{never seen} in the previous work of least-squares finite element methods.
  
  {
  	\color{red}
  	
  We fixed some errors in the published paper ``Chen, Gang; Yang, Fanyi; 
  Zhang, Zheyuan. Supercloseness Error Estimates for the Div Least-Squares 
  Finite Element Method on Elliptic Problems. J. Sci. Comput. 103 (2025), no. 3
  3, Paper No. 79. MR4894277''. We add some results in Lemma 4.4 and fix the proof in Theorem 4.3. We also enimilate the mesh requirement in Theorem 4.3.
}

\end{abstract}

\section{Introduction}\label{intro}

The least-squares finite element method (LSFEM) is a sophisticated
numerical technique for solving a variety of partial differential equations.
Applications include among others the second-order elliptic system
\cite{Pehlivanov1994least,MR2257124,Cai2002discrete,Cai2015div,MR2177148},
the Stokes \cite{Cai1997first,Bochev2013nonconforming,Liu2013hybrid} and the Navier-Stokes
equations \cite{Bochev1997analysis}, the
linear elasticity problem \cite{Cai2004least,Bramble2001least,Starke2011analysis}, the optimal control problem \cite{Fuhrer2023least,Bochev2006least} and the  Maxwell problem \cite{Bramble2005approximation,Jagalur2013galerkin,Li2022discontinuous}.
LSFEM with discontinuous elements has also been investigated in recent
years, such as \cite{Bensow2005discontinuous,Bensow2005div,Bochev2012locally,Bochev2013nonconforming,Li2019least,Li2019sequential}.
Different from Galerkin methods, LSFEM seeks the numerical solution by
minimizing the proper functional on a proper approximation space.
Then, the error estimation under the energy norm naturally follows from the
equivalence between the energy norm and the quadratic functional.
It is noticeable that the optimal $L^2$ convergence analysis is not trivial,
especially for the flux variable \cite{MR2177148}.

In this paper, we consider the div LSFEM for the following model problem: find $u\in H^1(\Omega)$
such that
\begin{subequations}\label{org}
  \begin{gather}
    -\nabla\cdot(\sigma\nabla u)-\omega^2\eta u =f\quad\text{in }\Omega, \\
    u                                           =0\quad\text{on }\Gamma_D \quad \text{and} \quad
    \zzy{\sigma\nabla u\cdot\bm n }                          =0\quad\text{on }\Gamma_N,
  \end{gather}
\end{subequations}
where $\sigma=\sigma(\bm x)>0$ and $\eta=\eta(\bm x)\neq 0$ are smooth functions, and $\omega^2$ is not the eigenvalue of the problem:
\begin{subequations}\label{eigen}
  \begin{gather}
    -\nabla\cdot(\sigma\nabla u)-\omega^2\eta u =0\quad\text{in }\Omega, \\
    u                                           =0\quad\text{on }\Gamma_D \quad \text{and} \quad
    \zzy{\sigma\nabla u\cdot\bm n }                          =0\quad\text{on }\Gamma_N.
  \end{gather}
\end{subequations}
\zzya{
  \begin{remark}\label{rem_hiwave}
    In this paper, the case of high wavenumber is not specifically addressed. For this case, the scheme need to be modified into a form related to wavenumber, see \cite{CHEN2017145}, which is a future topic.
  \end{remark}
}
Here $\Omega \subset \mathbb R^d,d=2,3$ is an open bounded domain with a Lipschitz
boundary $\partial \Omega = \Gamma_N \cup \Gamma_D$,
where $\Gamma_N $ and $ \Gamma_D$ are two disjoint parts of $\partial \Omega$.

We introduce the flux variable $\bm q=\sigma\nabla u$ to rewrite the
problem \eqref{org} into a first-order system:
\begin{subequations}\label{mixed-org}
  \begin{gather}
    {\sigma^{-1}\bm q-\nabla u} =\bm 0,\quad -\nabla \cdot\bm q-\omega^2\eta u = f \quad \text{in }\Omega, \\
    u  =0 \quad \text{on }\Gamma_D \quad\text{and}\quad
    \bm n\cdot\bm q =0 \quad \text{on }\Gamma_N.
  \end{gather}
\end{subequations}

We consider the div LSFEM for \eqref{mixed-org} with boundary conditions enforced on the finite elements spaces.
The scalar variable $u$ is approximated by an $H^1$\zzy{-conforming} finite element space of
degree $m$ ($m\ge 1$, denoted as $\mathcal P_m$), and the flux variable $\bm{q}$
can be approximated by the BDM elements of degree $k$ ($k\ge 1$,  denoted as $\bm{\mathcal{BDM}}_k$) or
the RT elements of degree $k$ (denoted as $\bm{\mathcal{RT}}_k$).

Let us review some \emph{state-of-the-art} results for the error estimates to LSFEM. In Cai \zzy{\textit{et al.}} (1994) \cite{MR1302685}, the authors gave the error estimates in the energy norm with a convergence rate of $\min(k, m)$ for the element $\bm{\mathcal{BDM}}_k/\mathcal{P}_m$. Moreover, with the help of \zzy{$H^3$} regularity of the elliptic problem, the $L^2$ {estimates} for flux and scalar variables were improved by Bochev and Gunzburger in \cite{MR2177148}. In Cai \zzy{\textit{et al.}} (2006) \cite{MR2257124}, for $\bm{\mathcal{RT}}_{k}/\mathcal{P}_{k+1}$, the authors proved the optimal $L^2$ error estimate for the scalar variable $u$ without \zzy{$H^3$} regularity, by constructing a proper dual problem. \zzy{In Ku (2011) \cite{Ku2011SharpLE}, author established $L^2$ error estimates for flux and scalar variables with $H^{1+\alpha},\alpha>1/2$ regularity and the optimal convergence rates can be obtained with only $H^2$ regularity. In Bernkopf \textit{et al.} (2023) \cite{Bernkopf2023}, authors also gave the optimal $L^2$ error estimates for the flux variable, the scalar variable and its gradient with the help of two \zzyb{special} operators for the flux , moreover, \zzyb{their} results require lower regularity for the flux variable \zzyb{compared} to the scalar variable. In F\"{u}hrer \textit{et al.} (2022) \cite{doi:10.1137/21M1457023}, the authors used some regularization operators to obtain the optimal estimate for the scalar variable under singular data. In Liang \textit{et al.} (2023) \cite{10.1007/s10915-023-02246-x}, authors gave an optimal estimate for the scalar variable with $H^2$ regularity using nonconforming finite elements.} All \emph{state-of-the-art} $L^2$ estimates are summarized in \Cref{curr_res}.

\begin{table}[H]\Large
  \renewcommand{\arraystretch}{1}
  \centering
  \resizebox{1.0\textwidth}{!}{
    \begin{tabular}{c|c|c|c|c|c}
      \Xhline{1pt}
      {$\bm P_h$}              & {$V_h$}                 & {$\|\bm q-\bm q_h\|_0$} & {$\|\nabla\cdot(\bm q-\bm q_h)\|_0$}
                               & {$\|u-u_h\|_0$}
                               & {$\|\nabla(u-u_h)\|_0$}                                                                                        \\
      \Xhline{1pt}
      $\bm{\mathcal{BDM}}_k$   & $\mathcal P_{k-1}$      & $\zzy{k}$               & $k-1$                                & $\zzy{k}$   & $k-1$ \\
      \hline
      $\bm{\mathcal{BDM}}_k$   & $\mathcal P_{k}$        & $\zzy{k+1}$             & $k$                                  & $\zzy{k+1}$ & $k$   \\

      \hline
      $\bm{\mathcal{BDM}}_{k}$ & $\mathcal P_{k+1}$      & $\zzy{k+1}$             & $k$                                  & $\zzy{k+1}$ & $k$   \\

      \hline

      $\bm{\mathcal{RT}}_k$    & $\mathcal P_{k-1}$      & $\zzy{k}$               & $k-1$                                & $\zzy{k}$   & $k-1$ \\
      \hline
      $\bm{\mathcal{RT}}_k$    & $\mathcal P_{k}$        & $\zzy{k+1}$             & $k$                                  & $\zzy{k+1}$ & $k$   \\

      \hline
      $\bm{\mathcal{RT}}_{k}$  & $\mathcal P_{k+1}$      & $k+1$                   & $k+1$                                & $k+2$       & $k+1$ \\

      \hline

      \Xhline{1pt}
    \end{tabular}}
  \caption{The state-of-art error estimates. \zzy{In cases where $V_h = P_1$, one can only achieve suboptimal rates of convergence for $\|u-u_h\|_0$ and $\|\nabla(u-u_h)\|_0$.}}
  \label{curr_res}
\end{table}

In this paper, we present a systematical analysis of the LSFEM with both $\bm{\mathcal{BDM}}_k/\mathcal P_m$
and $\bm{\mathcal{RT}}_k/\mathcal P_m$ for different pairs of $(k, m)$.  Compared to
the previous results, the optimal error estimates under the $L^2$ norm are reached
\zzy{without requiring} $H^3$ regularity \zzy{for the continuous solution of} the elliptic system in most cases, including
the common choices $k = m, k + 1 = m$. The main techniques \zzy{lie in establishing} the error estimates between finite element solutions and \zzy{specific types of projections:  elliptic projection and H(div) projection.} In addition, some \zzya{supercloseness} \zzy{results are derived in our analysis, which are \emph{never seen} in the previous work of LSFEM.}

\zzy{We briefly} summarize our optimal estimates and \zzya{supercloseness} results in \Cref{opt} and \Cref{sup}.
\begin{table}[H]\Large
  \renewcommand{\arraystretch}{1}
  \centering
  \resizebox{1.0\textwidth}{!}{
    \begin{tabular}{c|c|c|c|c|c}
      \Xhline{1pt}
      {$\bm P_h$}              & {$V_h$}                 & {$\|\bm q-\bm q_h\|_0$} & {$\|\nabla\cdot(\bm q-\bm q_h)\|_0$}
                               & {$\|u-u_h\|_0$}
                               & {$\|\nabla(u-u_h)\|_0$}                                                                                       \\
      \Xhline{1pt}
      $\bm{\mathcal{BDM}}_k$   & $\mathcal P_{k-1}$      & $ \zzy{(k+k_1)^*}$      & $k $                                 & $k $      & $k-1 $ \\
      \hline
      $\bm{\mathcal{BDM}}_k$   & $\mathcal P_{k}$        & $k+1$                   & $k$                                  & $k+1$     & $k$    \\

      \hline
      $\bm{\mathcal{BDM}}_{k}$ & $\mathcal P_{k+1}$      & $k+1$                   & $k$                                  & ${k+k_2}$ & $k+1$  \\

      \hline

      $\bm{\mathcal{RT}}_k$    & $\mathcal P_{k-1}$      & $\zzy{(k+k_1)^*}$       & $k$                                  & $k$       & $k-1$  \\
      \hline
      $\bm{\mathcal{RT}}_k$    & $\mathcal P_{k}$        & $k+1$                   & $k+1$                                & $k+1$     & $k$    \\

      \hline
      $\bm{\mathcal{RT}}_{k}$  & $\mathcal P_{k+1}$      & $k+1$                   & $k+1$                                & $k+2$     & $k+1$  \\

      \hline

      \Xhline{1pt}
    \end{tabular}}
  \caption{Optimal error estimates, {* means $H^3$ elliptic regularity is required, $k_1:=\min(k-2,1)$, $k_2:=\min(k,2)$}}
  \label{opt}
\end{table}

As shown in \Cref{opt}, {only two \zzy{cases} of optimal estimates are established with} \zzy{$H^3$} regularity. \zzy{Also, some of our optimal estimates are \zzyb{established} with a lower regularity \zzyb{for} the flux \zzyb{compared} to the scalar.} For the \zzya{supercloseness}, \zzy{$H^3$} regularity is only used for $\|\Pi_{V_h}u-u_h\|_0$ and one case of $\|\nabla(\Pi_{V_h}u-u_h)\|_0$.
\begin{table}[H]\Large
  \renewcommand{\arraystretch}{1.2}
  \centering
  \resizebox{1.0\textwidth}{!}{
    \begin{tabular}{c|c|c|c|c|c}
      \Xhline{1pt}
      {$\bm P_h$}                            & {$V_h$}                             & $\|\bm{\Pi}_{\bm P_h}\bm q-\bm q_h\|_0$ & {$\|\nabla\cdot(\bm{\Pi}_{\bm P_h}\bm q-\bm q_h)\|_0$}
                                             & {$\|\Pi_{V_h}u-u_h\|_0$}
                                             & {$\|\nabla(\Pi_{V_h}u-u_h)\|_0$}                                                                                                                                                               \\
      \Xhline{1pt}
      $\bm{\mathcal{BDM}}_k$                 & $\mathcal P_{k-1}$                  & $\zzy{(k+k_1)^*}$                       & $k$                                                    & $\zzy{(k+k_1)^*}$                    & $k$            \\
      \hline
      $\bm{\mathcal{BDM}}_k$                 & $\mathcal P_{k}$                    & $k+1$                                   & $k+1$                                                  & $\zzy{(k+k_2)^*}$                    & $k+1$          \\

      \hline
      $\bm{\mathcal{BDM}}_{k}$               & $\mathcal P_{k+1}$                  & $k+1$                                   & $k+2$                                                  & {$k+k_2$}                            & $k+1$          \\

      \hline

      $\bm{\mathcal{RT}}_k$                  & $\mathcal P_{k-1}$                  & $\zzy{(k+k_1)^*}$                       & $k$                                                    & $\zzy{(k+k_1)^*}$                    & $k$            \\
      \hline
      $\bm{\mathcal{RT}}_k$                  & $\mathcal P_{k}$                    & $k+1$                                   & $k+1$                                                  & $\zzy{(k+k_2)^*}$                    & $k+1$          \\

      \hline
      \multirow{2}{*}{$\bm{\mathcal{RT}}_k$} & \multirow{2}{*}{$\mathcal P_{k+1}$} & \multirow{2}{*}{${k+1}$}                & \multirow{2}{*}{$k +2$ }                               & \multirow{2}{*}{$\zzy{(k+2+k_3)^*}$} & {
      $\zzy{(k+2)^*}$ if $k>0$ }                                                                                                                                                                                                              \\
                                             &                                     &                                         &                                                        &                                      & $k+2$ if $k=0$ \\

      \hline
      \Xhline{1pt}
    \end{tabular}}
  \caption{\zzya{Supercloseness} error estimates, {* means $H^3$ elliptic regularity is required, $k_1:=\min(k-2,1)$, $k_2:=\min(k,2)$, $k_3:=\min(k,1)$}}
  \label{sup}
\end{table}
In \Cref{sup}, {$\Pi_{V_h}$ and $\bm{\Pi}_{\bm P_h}$, \zzy{as presented in \Cref{projection}, are critical for obtaining the \zzya{supercloseness} results.} For the cases of $\|\bm{\Pi}_{\bm P_h}\bm q-\bm q_h\|_0$, we only establish suboptimal and optimal approximation. \zzy{For the cases} of $\|\nabla\cdot(\bm{\Pi}_{\bm P_h}\bm q-\bm q_h)\|_0$, \zzya{supercloseness} can be obtained for $\bm{\mathcal{BDM}}_k/\mathcal P_{k}$, $\bm{\mathcal{BDM}}_k/\mathcal P_{k+1}$ and $\bm{\mathcal{RT}}_k/\mathcal P_{k+1}$. For the cases of $\|\Pi_{V_h}u-u_h\|_0$ and $\|\nabla(\Pi_{V_h}u-u_h)\|_0$, \zzya{supercloseness} are established except for one situation that $\Pi_{V_h}u-u_h$ with $\bm{\mathcal{BDM}}_k/\mathcal P_{k+1}$.}
Numerical results presented in \Cref{nume_test} confirm that our analysis is indeed sharp for most cases.

The paper is organized as follows. In the rest of this section, we introduce the notation used throughout the paper. In \Cref{section_2}, we present the scheme of LSFEM. \Cref{Primary_estimate} shows the primary error estimates of our scheme from the error estimates of some projections. In \Cref{err_dual}, we design several special dual arguments to obtain the $L^2$ error estimates of $\Pi_{V_h}u-u_h$, $\nabla\cdot(\bm\Pi_{\bm P_h}\bm q-\bm q_h)$, $\bm\Pi_{\bm P_h}\bm q-\bm q_h$ and $\nabla(\Pi_{V_h}u-u_h)$. \zzy{\Cref{nume_test} provides numerical experiments to confirm the theoretical analysis.}

\subsection{Notation}
For a bounded domain $\Lambda$ and $p>0$,  let $H^p(\Lambda)$ denote the Sobolev space of order $p$ on $\Lambda$, and $L^2(\Lambda)=H^0(\Lambda)$. \zzy{We denote} vector-valued functions and vector analogous of the Sobolev space by bold-face fonts, such as $\bm L^2(\Lambda) = [L^2(\Lambda)]^d,d=2,3$. On Sobolev space $H^p(\Lambda)$ and $[H^p(\Lambda)]^d$, we follow the corresponding \zzy{notation for} norm $\|\cdot\|_{p,\Lambda}$ and the \zzy{semi-norm $|\cdot|_{p,\Lambda}$}.
The inner products on domain $\Lambda$ and boundary $\partial \Lambda$ of $[L^2(\Lambda)]^r$ and $[L^2(\partial\Lambda)]^r$, $r=1,2,3$ are standard. Moreover, $\bm n_{\Lambda}$ denotes the outward unit normal vector to $\partial \Lambda$. For the case that $\Lambda=\Omega$, we omit the symbol $\Omega$, for example, $\|\cdot\|_{p}:=\|\cdot\|_{p,\Omega}$.

Next, we recall the spaces
  {\begin{gather*}
      \bm H({\rm div};\Omega)  :=\{\bm q\in \bm L^2(\Omega):\nabla\cdot\bm q\in L^2(\Omega)\},\\
      \bm H_N({\rm div};\Omega)  :=\{\bm q\in \bm H({\rm div};\Omega):\bm q\cdot\bm n =0\text{ on }\Gamma_N   \},\\
      H_D^1(\Omega)              :=\{
      u\in H^1(\Omega): u=0\text{ on }\Gamma_D
      \}.
    \end{gather*}}
Moreover, we denote by
\begin{align*}
  \bm H_N^{\gamma}({\rm div};\Omega)  :=\bm H_N({\rm div};\Omega)\cap \bm H^{\gamma}(\Omega), \qquad H_D^{\eta}(\Omega) :=H_D^1(\Omega) \cap H^{\eta}(\Omega),
\end{align*}
where $\gamma\ge0$, $\eta\ge 1$.
The space $H^{-1}_D(\Omega)$ denotes the dual space of $H^{1}_D(\Omega)$, with the negative norm
\begin{align*}
  \|u\|_{-1,D}=\sup\limits_{0\neq v\in H^1_D(\Omega)}\frac{(u,v)}{\|v\|_1} \quad \forall u\in H^{-1}_D(\Omega),
\end{align*}
\zzy{where $(\cdot,\cdot)$ denotes the duality pairing between $H^{-1}_D(\Omega)$ and $H^1_D(\Omega)$.} By the definition of the negative norm and the Poincaré inequality, it holds
\begin{align}\label{neg_norm}
  |(u,v)|\le C\|u\|_{-1,D}\| \nabla v\|_0 \quad \forall u\in H^{-1}_D(\Omega), v\in H_D^1(\Omega).
\end{align}
Throughout this paper, $C$ is a generic constant that may
change from line to line but independent of the mesh size.
\section{Least-Squares finite element methods}\label{section_2}
In this section, we present the scheme of LSFEM.
From \eqref{org}, we consider the problem with a more general right-hand side:
for any given $\bm g\in \bm L^2(\Omega)$ and $f\in L^2(\Omega)$, our goal is to find $(\bm q, u)$ such that
\begin{subequations}
  \begin{gather*}
    {\sigma^{-1}\bm q-\nabla u} =\bm g,\quad -\nabla \cdot\bm q-\omega^2\eta u = f \quad \text{in }\Omega, \\
    u  =0 \quad \text{on }\Gamma_D \quad\text{and}\quad
    \bm n\cdot\bm q =0 \quad \text{on }\Gamma_N,
  \end{gather*}
  in which $\omega^2$ is not the eigenvalue of \eqref{eigen}.
\end{subequations}

The bilinear form $a(\cdot;\cdot)$ on spaces $[\bm H_N({\rm div};\Omega)\times H_{D}^1(\Omega)]^2$ is defined as
\begin{align}\label{defa}
  a(\bm q,u;\bm p, v):=(\sigma(\sigma^{-1}\bm q-\nabla u),\sigma^{-1}\bm p-\nabla v)
  +(\nabla\cdot\bm q+\omega^2\eta u,\nabla\cdot\bm p+\omega^2\eta v).
\end{align}
Then the LSFEM reads: find $(\bm q,u)\in \bm H_N({\rm div};\Omega)\times H_{D}^1(\Omega)$ such that
\begin{align}\label{lsorg}
  a(\bm q,u;\bm p,v)=(\sigma\bm  g,\sigma^{-1}\bm p-\nabla v )-(f,\nabla\cdot\bm p+\omega^2\eta v)
\end{align}
for any $(\bm p,v)\in \bm H_N({\rm div};\Omega)\times H_{D}^1(\Omega)$.

The energy norm corresponding to the least-squares principle on space $\bm H_N({\rm div};\Omega)\times H_{D}^1(\Omega)$ is defined as
\begin{align*}
  \|(\bm q,u)\|^2:= \|\sigma^{-\frac 1 2}\bm q\|^2_{0}
  +\|\nabla\cdot\bm q\|^2_{0}
  +\|\sigma^{\frac 1 2}\nabla u\|^2_{0}
  +\|\omega^2\eta u\|_0^2
  .
\end{align*}
{The following result is standard for the LSFEM.}

\begin{lemma}[{\cite[Theorem 3.1]{MR1302685}},{\cite[Theorem 3.1]{Zhang2023}}] \label{lemma-a} The continuity property {and the coercivity} holds
  \begin{gather}
    |a(\bm q,u;\bm p,v)|  \le  	C\|(\bm q,u)\|\|(\bm p,v)\| \quad\forall (\bm q,u),(\bm p, v) \in \bm H_N({\rm div};\Omega)\times H_{D}^1(\Omega),\label{con1}\\
    a(\bm q,u;\bm q,u)  \ge C\|(\bm q,u)\|^2 \quad\forall (\bm q,u)\in \bm H_N({\rm div};\Omega)\times H_{D}^1(\Omega).\label{con2}
  \end{gather}
\end{lemma}

For the finite element approximation, let $\mathcal T_h$ represent a collection of shape-regular simplexes  $\{T\}$ that partition the domain $\Omega$, with mesh size $h=\max_{T\in\mathcal T_h}h_T$, where $h_T$ is the diameter of the circumcircle of $T$.

We choose the space $\bm P_h\times V_h\subset \bm H_N({\rm div};\Omega)\times H^1_D(\Omega)$, then the LSFEM can be expressed as follows: find $(\bm q_h,u_h)\in \bm P_h\times V_h$ such that
\begin{align}\label{lsfem}
  a(\bm q_h,u_h;\bm p_h,v_h)=(\sigma \bm g,\sigma^{-1}\bm p_h-\nabla v_h )-(f,\nabla\cdot\bm p_h +\omega^2\eta v_h)
\end{align}
for any $(\bm p_h,v_h)\in \bm P_h\times V_h$.

The following two results are trivial but useful in our analysis.
\begin{lemma}[Orthogonality]\label{lemma-or} Let $(\bm q,u)\in \bm H_N^{s}({\rm div};\Omega)\times H^1_D(\Omega)$ with $s>1/2$ and $(\bm q_h,u_h)\in \bm P_h\times V_h$ be the solution of \eqref{lsorg} and \eqref{lsfem}, respectively, then
  \begin{align}\label{orth}
    a(\bm q-\bm q_h,u-u_h;\bm p_h,v_h)=0 \quad\forall(\bm p_h,v_h)\in \bm P_h\times V_h.
  \end{align}
\end{lemma}

\begin{lemma}[Stability] Let $(\bm q_h,u_h)\in \bm P_h\times V_h$ be the solution of \eqref{lsfem}. Then the stability  holds
  \begin{align}\label{sta}
    \|(\bm q_h,u_h)\|\le C(\|\bm g\|_0 + \|f\|_0).
  \end{align}
\end{lemma}

\section{Primary error estimates}\label{Primary_estimate}
In this section, we first introduce the $L^2$ projection and the elliptic projection, then some fundamental estimates can be obtained from these projections.
\subsection{The spaces}\label{space}
Let $\mathcal P_k(\Lambda)(k\ge 1)$ denote the set of all polynomials with a maximum \zzy{degree  $k$} on $\Lambda$.
For $m\ge 1$, the space of \zzy{$H_D^1$-conforming piecewise $\mathcal{P}_m$ functions} is of form
\begin{align}
  P_m=\{v_h\in H_D^1(\Omega): v_h|_T\in \mathcal P_{m}(T)\quad \forall T\in\mathcal T_h\},
\end{align}
which is denoted by $V_h$.
\zzy{The spaces of $\bm{\mathcal{BDM}}_k$ element is defined as}
\zzy{\begin{align*}
  \bm{\mathcal{BDM}}_k=\{
  \bm p_h\in \bm H_N({\rm div};\Omega): \bm p_h|_T\in [\mathcal P_{k}(T)]^d \quad \forall T\in\mathcal T_h\},
\end{align*}
and for $k\ge 0$, the spaces of $\bm{\mathcal{RT}}_k$ element is defined as
\begin{align*}
  \bm{\mathcal{RT}}_k=\{
  \bm p_h\in \bm H_N({\rm div};\Omega): \bm p_h|_T\in \bm{\mathcal R}_{k}(T)  \quad\forall T\in\mathcal T_h\},
\end{align*}
where $\bm{\mathcal R}_k(T)=[\mathcal P_k(T)]^d+\bm x\mathcal P_k(T)$.} For simplicity, we denote $\bm{\mathcal{BDM}}_k$ or $\bm{\mathcal{RT}}_k$ by $\bm P_h$.

\subsection{Projection}\label{projection}

Let $\bm{\Pi}_{\bm P_h}$ denotes the $H(\rm div)$ projection onto \zzy{$\bm P_h$.}
\subsubsection{$L^2$ projection}
The following definitions and approximation results are quite standard.
\begin{lemma}[{\cite[Corollary 2.1]{Boffi2008MixedFE}}] {For $v\in L^2(\Omega)$, let $\Pi_k^ov \in \mathcal P_k(T)$ be the $L^2$ projection such that}
  \begin{align*}
    (\Pi_k^ov,w_h)_{T}=(v,w_h)_{T} \quad\forall w_h\in \mathcal P_k(T),
  \end{align*}
  with the property
  \begin{align}\label{err1}
     & \left\|v-\Pi_k^ov\right\|_{0, T} \le C h_T^s|v|_{s, T}, \quad 0 \leq s \leq k+1 \quad\forall v \in H^s(T).
  \end{align}
\end{lemma}

The following lemma can be found in {\cite[Lemma 3.5]{Boffi2008MixedFE}} and {\cite[Equation 2.13]{Brezzi1987}}.
\begin{lemma} \label{lemma-commu}For any $\bm v\in [H^{\frac 1 2+s}(\Omega)]^d$ with $s>0$, it holds
  \begin{align}\label{lem3.4}
    \Pi_{\ell}^o(\nabla\cdot\bm v)=\nabla\cdot( \bm{\Pi}_{\bm P_h}\bm v),\qquad \ell=\left\{
    \begin{aligned}
       & k - 1, & \bm P_h = \bm{\mathcal{BDM}}_{k}, \\
       & k,     & \bm P_h = \bm{\mathcal{RT}}_{k}.
    \end{aligned}
    \right.
  \end{align}
\end{lemma}

\subsubsection{Elliptic projection}\label{ell_pro}
The elliptic projection $\Pi_{V_h}: H_D^1(\Omega)\to V_h$ is defined as: for any given $w\in H_D^1(\Omega)$, find $\Pi_{{V_h}}w\in V_h$ such that
\begin{align}\label{ell_p_1}
  (\sigma\nabla \Pi_{V_h}w, \nabla v_h)=(\sigma\nabla w,\nabla v_h) \quad \forall v_h\in V_h.
\end{align}

To prove the $L^2$ norm error estimates in dual arguments, we recall some regularity
results for the \zzy{elliptic system}: given any $\theta
  \in L^2(\Omega)$, find $z \in H^1(\Omega)$ such that
\begin{align}\label{elliptic}
  \zzy{-}\nabla\cdot(\sigma\nabla z)-\omega^2\eta z  =\theta\quad  \text{in }\Omega, \quad z                                           =0\quad       \text{on }\Gamma_D \quad \text{and}\quad
  \bm n \cdot\nabla z                         =0\quad        \text{on }\Gamma_N.
\end{align}
It is well-known that the above problem admits a
unique solution $z \in H_D^{1}(\Omega)$ under the condition that $\omega^2$ is not the eigenvalue of the problem \eqref{eigen}, and there holds \cite[Theorem 18.13]{Dauge1988elliptic},
\begin{align}\label{preg1}
  \|z\|_{1+\varepsilon} & \le C\|\theta\|_{-1+\varepsilon}
\end{align}
for some $\varepsilon\in (\frac 1 2,1]$ depends on the geometry of $\Omega$.
The right-hand side of \eqref{elliptic} can also be understood
in $H^{-1}$, i.e.,
\begin{align}\label{preg3}
  \|z\|_{1}\le C\|\theta\|_{-1,D}.
\end{align}
We also need the regularity result for the derivatives of order
higher than two. By \cite[Chapter 1]{Costabel2012},
for $\theta\in H_D^1(\Omega)$, there holds
\zzy{\begin{align}\label{preg2}
    \|z\|_{1+\beta}\le C\|\theta\|_{-1 + \beta},\quad \text{for some } \beta\in \left(\frac 1 2,2\right].
  \end{align}}
\zzy{\begin{remark}
    It is well known that $\varepsilon=1$ holds in \eqref{preg1} for the case $\Gamma_D=\partial\Omega$ and $\Omega$ is convex \cite{Dauge1988elliptic}. Moreover,
    $\beta$ can reach $2$ in \eqref{preg2} when $\Gamma_D=\partial\Omega$, $d=2$ and the largest angle of $\partial
      \Omega$ is strictly less than 90 degrees \cite[Equation (1.4)]{Costabel2012}.
  \end{remark}}
\zzy{Next, we provide regularity of the solution to the continuous problem. First, we can obtain from \eqref{preg1} that
  \begin{align}\label{reg_con_1}
    \|u\|_{1+\varepsilon} + \|\bm q\|_{\varepsilon}\leq C\| f\|_{-1+\varepsilon}\quad \text{for some }\varepsilon\in \left(\frac{1}{2},1\right].
  \end{align}
  Second, by \eqref{preg2}
  \begin{align}\label{reg_con_2}
    \|u\|_{1+\beta} + \|\bm q\|_{\beta}\leq C\|f\|_{-1 + \beta}\quad \text{for some }\beta\in \left(\frac{1}{2},2\right].
  \end{align}
}
\begin{remark}\label{rem_3.5} \zzy{It is noted that $H^3$ regularity is needed only} in the case that
  $k=m+1\ge 3$ to derive the optimal error estimates for $\|\bm{\Pi}_{\bm P_h}\bm q-\bm q_h\|_0$, and to derive \zzya{supercloseness} results for some other cases. In the case $k=m$ and $k+1=m$, the optimal error estimates for  $\|\bm{\Pi}_{\bm P_h}\bm q-\bm q_h\|_0$ are obtained \zzy{without requiring $H^3$ regularity.}
\end{remark}

Based on the regularity of \eqref{elliptic}, the proof of the following approximation properties are quite standard.
\begin{lemma} \label{approximation-elliptic}The following error estimates hold when $w\in H_D^1(\Omega)$:
  \begin{align*}
    \|\sigma^{\frac 1 2}\nabla(\Pi_{V_h}w-w )\|_0 & \le \inf_{v_h\in V_h}\|\sigma^{\frac 1 2}\nabla(v_h-w )\|_{0},       \\
    \|\Pi_{V_h}w-w\|_0                            & \le Ch^{\varepsilon}\|\sigma^{\frac 1 2}\nabla(\Pi_{V_h}w-w )\|_0,   \\
    \|\Pi_{V_h}w-w\|_{-1,D}                       & \le Ch^{\min(\beta,m)}\|\sigma^{\frac 1 2}\nabla(\Pi_{V_h}w-w )\|_0.
  \end{align*}
\end{lemma}

\subsection{Error estimates}
We first present the following result, which plays a key role in the error estimations.
\begin{lemma}
  Let $(\bm q,u)\in \bm H_N^{s}({\rm div};\Omega)\times H^1_D(\Omega)$ with $s>1/2$, then for any $(\bm p_h,v_h)\in \bm P_h\times V_h$, it holds
  \begin{align}\label{pi-error}
    \begin{split}
       & a(\bm{\Pi}_{\bm P_h}\bm q-\bm q,\Pi_{V_h}u-u;\bm p_h, v_h)                                                            \\
       & \quad =(\sigma^{-1}(\bm{\Pi}_{\bm P_h}\bm q-\bm q ),\bm p_h)
      -(\bm{\Pi}_{\bm P_h}\bm q-\bm q,\nabla v_h)+(\Pi_{V_h}u-u ,\nabla\cdot\bm p_h)                                           \\
       & \quad\quad-\omega^2(\bm{\Pi}_{\bm P_h}\bm q-\bm q ,\nabla(\eta v_h) )+\omega^2(\eta(\Pi_{V_h}u-u),\nabla\cdot\bm p_h) \\
       & \quad\quad+\omega^4(\eta^2(\Pi_{V_h}u-u),v_h).
    \end{split}
  \end{align}
\end{lemma}

\begin{proof}
  It follows from the definition of $a(\cdot;\cdot)$ in \eqref{defa} that
  \begin{align}\label{3.3_a}
    \begin{split}
       & a(\bm{\Pi}_{\bm P_h}\bm q-\bm q,\Pi_{V_h}u-u;\bm p_h, v_h)       \\
       & \quad =(\sigma^{-1}(\bm{\Pi}_{\bm P_h}\bm q-\bm q ),\bm p_h)
      -(\bm{\Pi}_{\bm P_h}\bm q-\bm q,\nabla v_h)
      -(\nabla(\Pi_{V_h}u-u ),\bm p_h)                                    \\
       & \quad\quad
      +(\sigma\nabla(\Pi_{V_h}u-u ),\nabla v_h )
      +(\nabla\cdot( \bm{\Pi}_{\bm P_h}\bm q-\bm q ) ,\nabla\cdot\bm p_h) \\
       & \quad\quad
      +\omega^2(\nabla\cdot( \bm{\Pi}_{\bm P_h}\bm q-\bm q ),\eta v_h )
      +\omega^2(\eta(\Pi_{V_h}u-u),\nabla\cdot\bm p_h)                    \\
       & \quad\quad+\omega^4(\eta^2(\Pi_{V_h}u-u),v_h).
    \end{split}
  \end{align}
  The equation \eqref{pi-error} follows from \eqref{3.3_a}, the orthogonality \eqref{ell_p_1}, \Cref{lemma-commu}, the orthogonality of $\Pi_{\ell}^o$ and integration by parts.
\end{proof}

Now, we are ready to prove the following estimate, which will be used in our primary error analysis.
\begin{lemma}\label{lem_3.7}  Let $(\bm q,u)\in \bm H_N^{s}({\rm div};\Omega)\times H^1_D(\Omega)$ with $s>1/2$, then
  \begin{align}\label{E-error}
    \begin{split}
       & |a(\bm{\Pi}_{\bm P_h}\bm q-\bm q,\Pi_{V_h}u-u;\bm p_h, v_h)| \\
       & \quad\le C(
      \|\sigma^{-\frac 1 2}(\bm\Pi_{\bm P_h}\bm q-\bm q)\|_0
      +\|\Pi_{V_h}u-u\|_0
      )\|(\bm p_h, v_h)\|
    \end{split}
  \end{align}
  holds for any $(\bm p_h,v_h)\in \bm P_h\times V_h$.
\end{lemma}
\zzy{
  \begin{proof}
    \zzyb{Inequality} \ref{E-error} follows from the Cauchy-Schwarz inequality and \eqref{pi-error} directly.
  \end{proof}
}

From \Cref{lem_3.7} and the properties of our bilinear form in \Cref{lemma-a}, the estimate \zzy{in} energy norm can be established.

\begin{theorem}[\zzy{Error estimate} in energy norm]\label{the_3.10}{For $(\bm q,u)\in \bm H_N^{s}({\rm div};\Omega)\times H^1_D(\Omega)$ with $s>1/2$ and $(\bm q_h,u_h)\in \bm P_h\times V_h$ be the solution of \eqref{lsorg} and \eqref{lsfem}, respectively, there holds}
  \begin{align*}
    \|(\bm\Pi_{\bm P_h}\bm q-\bm q_h, \Pi_{V_h} u-u_h)\| & \le C(
    \|\sigma^{-\frac 1 2}(\bm\Pi_{\bm P_h}\bm q-\bm q)\|_0
    +\|\Pi_{V_h}u-u\|_0
    ).
  \end{align*}
\end{theorem}
\zzy{\begin{proof}
    It follows from \eqref{con2}, the orthogonality \eqref{orth} and \zzy{the estimate} \eqref{E-error} that
    \zzyb{\begin{align*}
       & \|(\bm\Pi_{\bm P_h}\bm q-\bm q_h, \Pi_{V_h} u-u_h)\|^2                                                 \\
       & \qquad  \le Ca(\bm\Pi_{\bm P_h}\bm q-\bm q_h, \Pi_{V_h} u-u_h;\bm\Pi_{\bm P_h}\bm q-\bm q_h,  \Pi_{V_h}u-u_h) \\
       & \qquad =   Ca(\bm\Pi_{\bm P_h}\bm q-\bm q_h, \Pi_{V_h} u-u_h;\bm\Pi_{\bm P_h}\bm q-\bm q,  \Pi_{V_h}u-u)      \\
       & \qquad \le C(
      \|\sigma^{-\frac 1 2}(\bm\Pi_{\bm P_h}\bm q-\bm q)\|_0
      +\|\Pi_{V_h}u-u\|_0
      )
      \|(\bm\Pi_{\bm P_h}\bm q-\bm q_h, \Pi_{V_h} u-u_h)\|,
    \end{align*}}
    which completes our proof.
  \end{proof}}
\zzy{Then}, we present error estimates in certain cases, which will be used in upcoming dual arguments.
\begin{remark}\label{rem_3.10} By the approximation properties of $\bm H({\rm div};\Omega)$ projection \cite[Lemma 3.5, Lemma 3.8]{Chen2016RobustGD} and \Cref{approximation-elliptic}, it holds
  \zzyc{
    \begin{align*}
    \|(\bm\Pi_{\bm P_h}\bm q-\bm q_h, \Pi_{V_h} u-u_h)\|\le C 
    (\|\bm q\|_{0}+\|u\|_{1})
  \end{align*}
  and}
  \begin{align*}
    \|(\bm\Pi_{\bm P_h}\bm q-\bm q_h, \Pi_{V_h} u-u_h)\|\le Ch^{\varepsilon}
    (\|\bm q\|_{\varepsilon}+\|u\|_{1})
  \end{align*}
  for any $(\bm q, u)\in \bm H^{\varepsilon}_N({\rm div};\Omega)\times H^{1+\varepsilon}_D(\Omega)$.
  Therefore, by the definition of the energy norm, we arrive at
  \zzyc{\begin{align}
    \begin{split}
       & \|\sigma^{-\frac{1}{2}}(\bm{\Pi}_{\bm P_h}\bm q-\bm q_h)\|_0
      +\|\nabla\cdot(\bm{\Pi}_{\bm P_h}\bm q-\bm q_h)\|_0
      +\|\sigma^{\frac{1}{2}}\nabla(\Pi_{V_h}u-u_h)\|_0               \\
       & \qquad \le C 
      (\|\bm q\|_{0}+\|u\|_{1})\label{es10}
    \end{split}
  \end{align}
  and}
  \begin{align}
    \begin{split}
       & \|\sigma^{-\frac{1}{2}}(\bm{\Pi}_{\bm P_h}\bm q-\bm q_h)\|_0
      +\|\nabla\cdot(\bm{\Pi}_{\bm P_h}\bm q-\bm q_h)\|_0
      +\|\sigma^{\frac{1}{2}}\nabla(\Pi_{V_h}u-u_h)\|_0               \\
       & \qquad \le Ch^{\varepsilon}
      (\|\bm q\|_{\varepsilon}+\|u\|_{1}).\label{es11}
    \end{split}
  \end{align}
  Moreover, combining the approximation properties for $\bm\Pi_{\bm P_h}$ and $\Pi_{V_h}$, yields
  \begin{align}
    \begin{split}
       & \|\sigma^{-\frac{1}{2}}(\bm q-\bm q_h)\|_0
      +\|\sigma^{\frac{1}{2}}\nabla(u-u_h)\|_0\le Ch^{\varepsilon}
      (\|\bm q\|_{\varepsilon}+\|u\|_{1+\varepsilon}).\label{es12}
    \end{split}
  \end{align}
\end{remark}
\begin{remark}\label{rem_3.11} For the case $k\ge 1$, we set $a=\min(\beta+\varepsilon-1,2\varepsilon)$ and $b=\max(1,\min(\beta,1+\varepsilon))$,
  \begin{align*}
    \|(\bm\Pi_{\bm P_h}\bm q-\bm q_h, \Pi_{V_h} u-u_h)\|\le Ch^{a}
    (\|\bm q\|_{a}+\|u\|_{b})
  \end{align*}
  for any $(\bm q, u)\in \bm H^{a}_N({\rm div};\Omega)\times H^{b}_D(\Omega)$.
  Therefore, it follows from the definition of energy norm that
  \begin{align}
    \begin{split}
       & \|\sigma^{-\frac{1}{2}}(\bm{\Pi}_{\bm P_h}\bm q-\bm q_h)\|_0
      +\|\nabla\cdot(\bm{\Pi}_{\bm P_h}\bm q-\bm q_h)\|_0
      +\|\sigma^{\frac{1}{2}}\nabla(\Pi_{V_h}u-u_h)\|_0               \\
       & \qquad \le Ch^{a}
      (\|\bm q\|_{a}+\|u\|_{b}).\label{es21}
    \end{split}
  \end{align}
\end{remark}

\zzy{By employing} the elliptic projection, we are able to \zzy{derive the} following optimal estimates, which contain the current error estimates, and some new \zzya{supercloseness} estimates. It is noted that these results are established without \zzy{$H^3$} regularity for the elliptic problems.
\begin{remark}[Optimal estimates]\label{rem_3.12} \zzy{If $\Omega$ is convex} and $k=m$, the optimal error estimate holds
  \begin{align*}
    \|\sigma^{-\frac 1 2}(\bm\Pi_{\bm P_h}\bm q-\bm q_h)\|_0
    +\|\Pi_{V_h} u-u_h\|_0\le Ch^{k+1}(
    \|\bm q\|_{k+1}+\|u\|_{k+1})
  \end{align*}
  for $(\bm q, u)\in \bm H^{k+1}_N({\rm div};\Omega)\times H^{k+1}_D(\Omega)$.
  For the case $k=m+1${, the optimal error estimate holds}
  \begin{align*}
    \|\nabla\cdot(\bm{\Pi}_{\bm P_h}\bm q-\bm q_h)\|_0  \le  Ch^{k}(\|\bm q\|_{k+1}+\|u\|_{k}  )
  \end{align*}
  for any $(\bm q, u)\in \bm H^{k+1}_N({\rm div};\Omega)\times H^{k}_D(\Omega)$,
  and $k+1=m${, the optimal error estimate holds}
  \begin{align*}
    \|\sigma^{\frac{1}{2}}\nabla(\Pi_{V_h}u-u_h)\|_0  \le  Ch^{m}(\|\bm q\|_{m}+\|u\|_{m+1}  )
  \end{align*}
  for any $(\bm q, u)\in \bm H^{m}_N({\rm div};\Omega)\times H^{m+1}_D(\Omega)$.
\end{remark}
\begin{remark}[\zzya{Supercloseness} results]\label{rem_3.13}
  The \zzya{supercloseness} can be obtained \zzy{if $\Omega$ is convex} and $k\le m$ (for $\bm{\mathcal RT}$ element, it is optimal):
  \begin{align*}
    \|\nabla\cdot(\bm{\Pi}_{\bm P_h}\bm q-\bm q_h)\|_0  \le  Ch^{k+1}(\|\bm q\|_{k+1}+\|u\|_{k+1}  )
  \end{align*}
  for any $(\bm q, u)\in \bm H^{k+1}_N({\rm div};\Omega)\times H^{k+1}_D(\Omega)$,
  and $k\ge m$:
  \begin{align*}
    \|\sigma^{\frac{1}{2}}\nabla(\Pi_{V_h}u-u_h)\|_0  \le  Ch^{m+1}(\|\bm q\|_{m+1}+\|u\|_{m+1}  )
  \end{align*}
  for any $(\bm q, u)\in \bm H^{m+1}_N({\rm div};\Omega)\times H^{m+1}_D(\Omega)$.
\end{remark}

\section{Error estimates from dual arguments}\label{err_dual}
As shown in our numerical experiments, presented in \Cref{nume_test}, the error estimates for some cases can reach \zzya{supercloseness}, which can be partially illustrated by \Cref{the_3.10}.
In this section, we \zzyb{modify} the dual arguments in Cai \zzy{\textit{et al.}} (2010) \cite{MR2585180}, to prove \zzy{additional} \zzya{supercloseness} results \zzy{in comparison to those in} \Cref{the_3.10}.

\subsection{Error \zzy{estimates} for $\Pi_{V_h}u-u_h$}\label{dua_u}

First, some regularity results will be introduced.
\begin{lemma}\label{lem_4.1}
  We define the problem: find $z\in H_D^{1}(\Omega)$ such that
  \begin{subequations}\label{lem_4.1_1}
    \begin{gather}
      -\nabla\cdot(\sigma\nabla z)-\omega^2\eta z  =\Pi_{V_h}u-u_h \quad  \text{in }\Omega,\\
      z                                            =0\quad               \text{on }\Gamma_N \quad\text{and}\quad
      \bm n\cdot\nabla z                           =0\quad               \text{on }\Gamma_D,
    \end{gather}
  \end{subequations}
  and the problem: find $(\bm \Phi,\Psi)\in \bm H_N({\rm div};\Omega)\times H^1_D(\Omega)$ such that
  \begin{align}\label{dua_u_1}
    \sigma^{-1}\bm \Phi-\nabla \Psi         =-\nabla z,\quad
    -\nabla \cdot\bm \Phi-\omega^2\eta\Psi  = -z\quad\text{in }\Omega.
  \end{align}
  Then it holds
  \begin{align}\label{dua_u_12}
    \|\bm\Phi\|_{\varepsilon}+\|\nabla\cdot\bm\Phi\|_{1+\varepsilon}+\|\Psi\|_{1+\varepsilon}+\|z\|_{1+\varepsilon}\le C\|\Pi_{V_h}u-u_h\|_0,
  \end{align}
  and
  \begin{align}\label{dua_u_13}
    \|\bm\Phi\|_{\beta}\le C\|\Pi_{V_h}u-u_h\|_0.
  \end{align}
\end{lemma}
\begin{proof}
  By a direct calculation, \zzy{it holds in the distribution sense between $H_D^1$ and $H_D^{-1}$ that}
  \begin{align}
     & \quad-\nabla\cdot(\sigma\nabla\Psi)-\omega^2\eta \Psi=-z-\nabla\cdot(\sigma\nabla z),\label{lem4.1_1} \\
     & -\nabla\cdot(\sigma\nabla(\Psi-z))-\omega^2\eta (\Psi-z)=(\omega^2\eta-1)z.\label{lem4.1_2}
  \end{align}
  We find from \eqref{lem_4.1_1}, the regularity \eqref{preg1} and the Poincaré inequality that
  \zzy{$$ \|z\|_{1+\varepsilon} \le C\|\Pi_{V_h}u-u_h\|_0.$$}
  Therefore, the triangle inequality, the regularity assumption \eqref{preg1} and \eqref{lem4.1_1} imply
  \begin{align}\label{lem4.1_3}
    \begin{split}
      \|\Psi\|_{1+\varepsilon} \le C\|z+\nabla\cdot(\sigma\nabla z)\|_{-1+\varepsilon} \le C(\|z\|_{-1+\varepsilon}+ \|z\|_{1+\varepsilon})\le C\|\Pi_{V_h}u-u_h\|_0.
    \end{split}
  \end{align}
  Likewise, using the regularity assumption \eqref{preg2}, \eqref{lem4.1_2} and \eqref{dua_u_1} yields
  \begin{align*}
    \|\bm\Phi\|_{\beta}	\le C\|\Psi-z\|_{1+\beta}\le C\|z\|_1\le  C\|\Pi_{V_h}u-u_h\|_0,
  \end{align*}
  which leads to \eqref{dua_u_13}.
  The estimate \eqref{dua_u_12} follows from the regularity assumption \eqref{preg1}, \eqref{dua_u_1} and \eqref{lem4.1_3}.
\end{proof}

It is noted that \eqref{dua_u_13}, which comes from a higher regularity \eqref{preg2}, is established to obtain the \zzya{supercloseness} estimate.

We define the discrete problem corresponding to \eqref{dua_u_1}: find $(\bm \Phi_h,\Psi_h)\in \bm P_h\times V_h$ such that
\begin{align}\label{dis_4}
  a(\bm \Phi_h,\Psi_h;\bm p_h,v_h)=(-\sigma\nabla z,\sigma^{-1}\bm p_h-\nabla v_h)+(z,\nabla\cdot\bm p_h+\omega^2\eta v_h)
\end{align}
for any $(\bm p_h,v_h)\in \bm P_h\times V_h$.
\zzy{From \Cref{rem_3.10}, we can obtain
  \begin{align}
    \begin{split}
       & \|\sigma^{-\frac{1}{2}}(\bm{\Pi}_{\bm P_h}\bm \Phi-\bm \Phi_h)\|_0
      +\|\nabla\cdot(\bm{\Pi}_{\bm \Phi_h}\bm q-\bm \Phi_h)\|_0
      +\|\sigma^{\frac{1}{2}}\nabla(\Pi_{V_h}\Psi-\Psi_h)\|_0               \\
       & \qquad \le Ch^{\varepsilon}
      (\|\bm \Phi\|_{\varepsilon}+\|\Psi\|_{1}),\label{dis_4_ineq1}
    \end{split} \\
     & \|\sigma^{-\frac{1}{2}}(\bm \Phi-\bm \Phi_h)\|_0
    +\|\sigma^{\frac{1}{2}}\nabla(\Psi-\Psi_h)\|_0\le Ch^{\varepsilon}
    (\|\bm \Phi\|_{\varepsilon}+\|\Psi\|_{1+\varepsilon}).\label{dis_4_ineq2}
  \end{align}
  Moreover, for $k>1$, by \Cref{rem_3.11}
  \begin{align}
    \begin{split}
       & \|\sigma^{-\frac{1}{2}}(\bm{\Pi}_{\bm P_h}\bm \Phi-\bm \Phi_h)\|_0
      +\|\nabla\cdot(\bm{\Pi}_{\bm P_h}\bm \Phi-\bm \Phi_h)\|_0
      +\|\sigma^{\frac{1}{2}}\nabla(\Pi_{V_h}\Psi-\Psi_h)\|_0               \\
       & \qquad \le Ch^{a}
      (\|\bm \Phi\|_{a}+\|\Psi\|_{b}),\label{dis_4_ineq3}
    \end{split}
  \end{align}
  where $a=\min(\beta+\varepsilon-1,2\varepsilon)$ and $b=\max(1,\min(\beta,1+\varepsilon))$.
}

Next, we present two estimates for $a(\bm{\Pi}_{\bm P_h}\bm q-\bm q,\Pi_{V_h}u-u;\bm \Phi_h, \Psi_h)$, derived from \zzy{above from above three estimates}.

\begin{lemma}\label{lemma_a1}  The following estimate holds
  \begin{align}\label{es_a_1}
    \begin{split}
       & |a(\bm{\Pi}_{\bm P_h}\bm q-\bm q,\Pi_{V_h}u-u;\bm \Phi_h, \Psi_h)|                                                                                                                    \\
       & \quad\le Ch^{\varepsilon}(\|\bm\Phi\|_{\varepsilon}+\|\Psi\|_{1+\varepsilon})\|\sigma^{-\frac 1 2}(\bm{\Pi}_{\bm P_h}\bm q-\bm q)\|_0+ C\|(\bm\Phi_h,\Psi_h)\|\cdot\|\Pi_{V_h}u-u\|_0 \\
       & \quad\quad
      + Ch^{1+\min(\ell,\varepsilon)}
      (
      \omega^2	\|\Psi\|_{1+\min(\ell,\varepsilon)}
      +\|z\|_{1+\min(\ell,\varepsilon)}
      )	\|
      {\Pi}_{\ell}^o\nabla\cdot\bm q-\nabla\cdot\bm q\|_0
    \end{split}
  \end{align}
  for $(\bm q,u)\in \bm H_N({\rm div};\Omega)\times H^1_D(\Omega)$ be the solution of \eqref{lsorg} and any $(\bm p_h,v_h)\in \bm P_h\times V_h$, \zzy{whrere $\ell$ is stated in \Cref{lemma-commu} and $\varepsilon$ is stated in \Cref{ell_pro}.}
\end{lemma}
\begin{proof} From \eqref{pi-error} and the fact 	$\sigma^{-1}\bm \Phi-\nabla \Psi=-\nabla z$, we obtain
  \begin{align}\label{a_1}
    \begin{split}
       & a(\bm{\Pi}_{\bm P_h}\bm q-\bm q,\Pi_{V_h}u-u;\bm \Phi_h, \Psi_h)       \\
       & \quad=(\sigma^{-1}(\bm{\Pi}_{\bm P_h}\bm q-\bm q ),\bm \Phi_h-\bm\Phi)
      -(\bm{\Pi}_{\bm P_h}\bm q-\bm q,\nabla (\Psi_h-\Psi))
      \\
       & \quad\quad+(\Pi_{V_h}u-u ,\nabla\cdot\bm \Phi_h)
      -\omega^2(\bm{\Pi}_{\bm P_h}\bm q-\bm q ,\nabla(\eta \Psi_h) )            \\
       & \quad\quad
      +\omega^2(\eta(\Pi_{V_h}u-u),\nabla\cdot\bm\Phi_h)
      +\omega^4(\eta^2(\Pi_{V_h}u-u),\Psi_h)
      +(\bm{\Pi}_{\bm P_h}\bm q-\bm q,\nabla z)                                 \\
       & \quad\quad =:\sum_{i=1}^7E_i.
    \end{split}
  \end{align}
  In above equation, $E_i$ can be estimated term by term  in details.
  By the Cauchy-Schwarz inequality and \zzy{\eqref{dis_4_ineq2}}, it holds
  \begin{align*}
    |E_1|+|E_2| & \le C(\|\sigma^{-\frac 1 2}(\bm\Phi_h-\bm\Phi)\|_0
    +\|\sigma^{\frac 1 2}\nabla(\Psi_h-\Psi)\|_0
    )\|\sigma^{-\frac 1 2}(\bm{\Pi}_{\bm P_h}\bm q-\bm q)\|_0                                                                                       \\
                & \le Ch^{\varepsilon}(\|\bm\Phi\|_{\varepsilon}+\|\Psi\|_{1+\varepsilon})\|\sigma^{-\frac 1 2}(\bm{\Pi}_{\bm P_h}\bm q-\bm q)\|_0.
  \end{align*}
  \zzy{The estimates} of $E_3$, $E_5$ and $E_6$ follow from the Cauchy-Schwarz inequality directly
  \begin{align*}
    |E_3|+|E_5|  \le C\|\nabla\cdot\bm \Phi_h\|_0\|\Pi_{V_h}u-u\|_0,\
    |E_6|  \le C\omega^2(\omega^2\|\eta\Psi_h\|_0) \|\Pi_{V_h}u-u\|_0.
  \end{align*}
  \zzy{We find from the integration by parts, the Cauchy-Schwarz inequality and \eqref{dis_4_ineq1} that}
  \begin{align*}
                                                                    & |E_4|                                                             \\                                                           & =\omega^2|(\nabla\cdot(\bm\Pi_{\bm P_h}\bm q-\bm q),
    \eta\Psi_h-\Pi_{\ell}^o(\eta\Pi_{V_h}\Psi))|                                                                                        \\
                                                                    & =\omega^2|(\nabla\cdot(\bm\Pi_{\bm P_h}\bm q-\bm q),
    \eta(\Psi_h-\Pi_{V_h}\Psi)+(\eta\Pi_{V_h}\Psi-\Pi_{\ell}^o(\eta\Pi_{V_h}\Psi)))|                                                    \\
                                                                    & \le \omega^2|(\sigma^{-\frac{1}{2}}(\bm\Pi_{\bm P_h}\bm q-\bm q),
    \sigma^{\frac{1}{2}}\nabla\eta(\Psi_h-\Pi_{V_h}\Psi))|                                                                              \\
                                                                    & \quad
    + \omega^2|(\nabla\cdot(\bm\Pi_{\bm P_h}\bm q-\bm q),
    \eta\Pi_{V_h}\Psi-\Pi_{\ell}^o(\eta\Pi_{V_h}\Psi))|
    \\                                                                    & \le C\omega^2\|\sigma^{-\frac{1}{2}}(\bm\Pi_{\bm P_h}\bm q-\bm q)\|_0\|\sigma^{\frac{1}{2}}\nabla( \Psi_h-\Pi_{V_h}\Psi  )\|_0 \\
                                                                    & \quad
    +C\omega^2\|\nabla\cdot(\bm\Pi_{\bm P_h}\bm q-\bm q)\|_0
    h^{1+\min(\ell,\varepsilon)}\|\Psi\|_{1+\min(\ell,\varepsilon)} & (\text{by \eqref{err1}})
    \\
                                                                    & \le C\omega^2h^{\varepsilon}
    (
    \|\bm\Phi\|_{\varepsilon}+
    \|\Psi\|_{1})
    \|\sigma^{-\frac 1 2}(\bm\Pi_{\bm P_h}\bm q-\bm q)\|_0          & (\text{by \eqref{es11}})
    \\
                                                                    & \quad+
    C\omega^2
    h^{1+\min(\ell,\varepsilon)}\|\Psi\|_{1+\min(\ell,\varepsilon)}
    \|
    {\Pi}_{\ell}^o\nabla\cdot\bm q-\nabla\cdot\bm q\|_0.            & (\text{by \eqref{lem3.4}})
  \end{align*}
  Using the Cauchy-Schwarz inequality, \eqref{err1} and \eqref{lem3.4} yields
  \begin{align*}
    |E_7| & =|(\nabla\cdot(\bm{\Pi}_{\bm P_h}\bm q-\bm q), z)|
    =|({\Pi}_{\ell}^o\nabla\cdot\bm q-\nabla\cdot\bm q, z-{\Pi}_{\ell}^oz)| \\
          & \le C	\|
    {\Pi}_{\ell}^o\nabla\cdot\bm q-\nabla\cdot\bm q\|_0h^{1+\min(\ell,\varepsilon)}\|z\|_{1+\min(\ell,\varepsilon)}.
  \end{align*}
  The estimate \eqref{es_a_1} follows from above results and the definition of energy norm.
\end{proof}

The second estimation requires $k\ge 1$ ($\mathcal{RT}_0$ is excluded).

\begin{lemma}\label{lemma_a2} For the case $k\ge 1$, the following estimate holds
  \begin{align}\label{E-error3}
    \begin{split}
      | & a(\bm{\Pi}_{\bm P_h}\bm q-\bm q,\Pi_{V_h}u-u;\bm \Phi_h, \Psi_h)| \\
        & \quad\le
      Ch^{\min(\beta+\varepsilon-1,2\varepsilon)}(\|\bm\Phi\|_{\beta}+\|\Psi\|_{1+\varepsilon})
      (\|\sigma^{-\frac 1 2}(\bm{\Pi}_{\bm P_h}\bm q-\bm q)\|_0             \\
        & \quad\quad+\|\Pi_{\ell}^o\nabla\cdot\bm q-\nabla\cdot\bm q\|_0)
      \\
        & \quad\quad+Ch^{\varepsilon}(
      \|\bm\Phi\|_{\varepsilon}+
      \|\Psi\|_{1+\varepsilon}
      )(h^{\varepsilon}\|\Pi_{V_h}u-u\|_{0}+\|\Pi_{V_h}u-u\|_{-1,D})
      \\
        & \quad \quad
      +Ch\|\Pi_{V_h}u-u\|_0
      \|\nabla\cdot\bm \Phi\|_1
      +
      C\|\Pi_{V_h}u-u\|_{-1,D}(\omega^2\|\nabla\Psi\|_0+\|\nabla z\|_0)     \\
        & \quad\quad
      +C \omega^4\|\nabla\Psi_h\|_0\|\Pi_{V_h}u-u\|_{-1,D}
      \\
        & \quad\quad
      +Ch^{1+\min(\ell,\varepsilon)}\|\Pi_{\ell}^o\nabla\cdot\bm q-\nabla\cdot\bm q\|_0
      (
      \omega^2\|\Psi\|_{1+\min(\ell,\varepsilon)}
      +	\|z\|_{1+\min(\ell,\varepsilon)}
      )
    \end{split}
  \end{align}
  for $(\bm q,u)\in \bm H_N({\rm div};\Omega)\times H^1_D(\Omega)$ be the solution of \eqref{lsorg} and any $(\bm p_h,v_h)\in \bm P_h\times V_h$, \zzy{whrere $\ell$ is stated in \Cref{lemma-commu} and $\varepsilon,\beta$ are stated in \Cref{ell_pro}.}
\end{lemma}
\begin{proof}
  Same as the proof of \Cref{lemma_a1}, we estimate \eqref{a_1} in terms.
  By the use of the Cauchy-Schwarz inequality and \zzy{the estimate \eqref{dis_4_ineq3}}, we obtain
  \begin{align*}
    |E_1| & \le C\|\sigma^{-\frac 1 2}(\bm\Phi_h-\bm\Phi)\|_0
    \|\sigma^{-\frac 1 2}(\bm{\Pi}_{\bm P_h}\bm q-\bm q)\|_0                                                                                                       \\
          & \le Ch^{\min(\beta+\varepsilon-1,2\varepsilon)}(\|\bm\Phi\|_{\min(\beta+\varepsilon-1,2\varepsilon)}+\|\Psi\|_{\max(1,\min(\beta,1+\varepsilon))})     \\
          & \quad\|\sigma^{-\frac 1 2}(\bm{\Pi}_{\bm P_h}\bm q-\bm q)\|_0                                                                                          \\
          & \le Ch^{\min(\beta+\varepsilon-1,2\varepsilon)}(\|\bm\Phi\|_{\beta}+\|\Psi\|_{1+\varepsilon})\|\sigma^{-\frac 1 2}(\bm{\Pi}_{\bm P_h}\bm q-\bm q)\|_0.
  \end{align*}
  Moreover, by \eqref{approximation-elliptic}
  \begin{align*}
    |E_2| & =|(\nabla\cdot(\bm{\Pi}_{\bm P_h}\bm q-\bm q),\Psi_h-\Psi)|                                                                                        \\
          & \le C\|\nabla\cdot(\bm{\Pi}_{\bm P_h}\bm q-\bm q)\|_0\|\Psi_h-\Psi\|_0                                                                             \\
          & \le  C\|\nabla\cdot(\bm{\Pi}_{\bm P_h}\bm q-\bm q)\|_0( \|\nabla(\Psi_h-\Pi_{V_h}\Psi)\|_0+\|\Pi_{V_h}\Psi-\Psi\|_0)                               \\
          & \le Ch^{\min(\beta+\varepsilon-1,2\varepsilon)}(\|\bm\Phi\|_{\min(\beta+\varepsilon-1,2\varepsilon)}+\|\Psi\|_{\max(1,\min(\beta,1+\varepsilon))}) \\
          & \quad\cdot\|\Pi_{\ell}^o\nabla\cdot\bm q-\nabla\cdot\bm q\|_0                                                                                      \\
          & \le Ch^{\min(\beta+\varepsilon-1,2\varepsilon)}(\|\bm\Phi\|_{\beta}+\|\Psi\|_{1+\varepsilon})\|\Pi_{\ell}^o\nabla\cdot\bm q-\nabla\cdot\bm q\|_0.
  \end{align*}
  To estimate $E_3$ and $E_5$, we apply \eqref{estp-key} given in \Cref{lem4.4}, which yields
  \begin{align*}
     & |E_3|+|E_5| \\
     & \quad \le
    Ch^{\varepsilon}(
    \|\bm\Phi\|_{\varepsilon}+
    \|\Psi\|_{1+\varepsilon}
    )(h^{\varepsilon}\|\Pi_{V_h}u-u\|_{0}+\|\Pi_{V_h}u-u\|_{-1,D})
    \\
     & \quad \quad
    +Ch\|\Pi_{V_h}u-u\|_0
    \|\nabla\cdot\bm \Phi\|_1
    +
    C\|\Pi_{V_h}u-u\|_{-1,D}(\omega^2\|\nabla\Psi\|_0+\|\nabla z\|_0).
  \end{align*}
  We find from \zzy{the estimate} of $E_2$, the Cauchy-Schwarz inequality and the error estimate \eqref{err1} that
  \begin{align*}
    |E_4| & =\omega^2|(\nabla\cdot(\bm\Pi_{\bm P_h}\bm q-\bm q),\eta(\Psi_h-\Psi)+(\eta\Psi-\Pi_{\ell}^o(\eta\Psi)))|  \\
          & \le C\omega^2(\|\nabla\cdot(\bm\Pi_{\bm P_h}\bm q-\bm q)\|_0\|\Psi_h-\Psi\|_0                              \\
          & \quad+       \|\Pi_{\ell}^o\nabla\cdot\bm q-\nabla\cdot\bm q\|_0  \|\eta\Psi-\Pi_{\ell}^o(\eta\Psi)\|_0  ) \\
          & \le
    C\omega^2h^{\min(\beta+\varepsilon-1,2\varepsilon)}(\|\bm\Phi\|_{\beta}+\|\Psi\|_{1+\varepsilon})\|\Pi_{\ell}^o\nabla\cdot\bm q-\nabla\cdot\bm q\|_0
    \\
          & \quad+
    C\omega^2h^{1+\min(\ell,\varepsilon)}\|\Pi_{\ell}^o\nabla\cdot\bm q-\nabla\cdot\bm q\|_0\|\Psi\|_{1+\min(\ell,\varepsilon)}.
  \end{align*}
  Using \eqref{neg_norm}, we obtain
  \begin{align*}
    |E_6|\le C\omega^4\|\nabla\Psi_h\|_0\|\Pi_{V_h}u-u\|_{-1,D}.
  \end{align*}
  Applying the Cauchy-Schwarz inequality and \eqref{err1}, it holds
  \begin{align*}
    |E_7| & =|(\nabla\cdot(\bm{\Pi}_{\bm P_h}\bm q-\bm q), z)|
    =|({\Pi}_{\ell}^o\nabla\cdot\bm q-\nabla\cdot\bm q, z-{\Pi}_{\ell}^oz)| \\
          & \le C h^{1+\min(\ell,\varepsilon)}\|
    {\Pi}_{\ell}^o\nabla\cdot\bm q-\nabla\cdot\bm q\|_0\|z\|_{1+\min(\ell,\varepsilon)}.
  \end{align*}
  The estimate \eqref{E-error3} is established by summing up all above results.
\end{proof}
\begin{lemma}\label{lem4.4}
  For the discrete problem \eqref{dis_4}, the following estimate holds
  \begin{subequations}
  \begin{align}
    \begin{split}\label{estp-key}
       & |(\Pi_{V_h}u-u,\nabla\cdot\bm\Phi_h)| \\
       & \quad\quad\le
      Ch^{\varepsilon}(
      \|\bm\Phi\|_{\varepsilon}+
      \|\Psi\|_{1+\varepsilon}
      )(h^{\varepsilon}\|\Pi_{V_h}u-u\|_{0}+\|\Pi_{V_h}u-u\|_{-1,D})
      \\
       & \quad\quad
      +Ch\|\Pi_{V_h}u-u\|_0
      \|\nabla\cdot\bm \Phi\|_1+C\|\Pi_{V_h}u-u\|_{-1,D}(\omega^2\|\nabla\Psi\|_0+\|\nabla z\|_0),\\
    \end{split}
  \end{align}
  \zzyc{\begin{align}
    \begin{split}\label{estp-key2}
       & |(\Pi_{V_h}u-u,\nabla\cdot\bm\Phi_h)| \\
       & \quad\quad\le
      C (
      \|\bm\Phi\|_{0}+
      \|\Psi\|_{1}
      )(h^{\varepsilon}\|\Pi_{V_h}u-u\|_{0}+\|\Pi_{V_h}u-u\|_{-1,D})
      \\
       & \quad\quad
      +Ch\|\Pi_{V_h}u-u\|_0
      \|\nabla\cdot\bm \Phi\|_1+C\|\Pi_{V_h}u-u\|_{-1,D}(\omega^2\|\nabla\Psi\|_0+\|\nabla z\|_0),\\
    \end{split}
  \end{align}}
  \end{subequations}
\end{lemma}
\begin{proof}
  It follows that
  \begin{alignat*}{2}
     & |(\Pi_{V_h}u-u,\nabla\cdot\bm \Phi_h)| =|(\Pi_{V_h}u-u,\nabla\cdot(\bm{\Phi}_h-\bm\Phi)-\omega^2\eta\Psi+z)|              & \text{(by \eqref{dua_u_1})}  \\
     & \qquad\le|(\Pi_{V_h}u-u,\nabla\cdot(\bm{\Phi}_h-\bm\Pi_{\bm P_h}\bm\Phi)) |                                                                              \\
     & \qquad\quad+|(\Pi_{V_h}u-u,\nabla\cdot(\bm\Pi_{\bm P_h}\bm{\Phi}-\bm\Phi)) | +|(\Pi_{V_h}u-u,-\omega^2\eta \Psi + z)|                                    \\
     & \qquad\le |(\Pi_{V_h}u-u,\nabla\cdot(\bm{\Phi}_h-\bm\Pi_{\bm P_h}\bm\Phi)) |+Ch\|\nabla\cdot\bm\Phi\|_1\|\Pi_{V_h}u-u\|_0 & (\text{by \eqref{err1}})     \\
     & \qquad\quad+C\|\Pi_{V_h}u-u\|_{-1,D}(\omega^2\|\nabla\Psi\|_0+\|\nabla z\|_0).                                            & \text{(by \eqref{neg_norm})}
  \end{alignat*}

  We define the dual problem: find $(\bm r,w)\in \bm H^1_N({\rm div};\Omega)\times H_D^{1}(\Omega)$ such that
  \begin{align}\label{dual_r_w}
    \sigma^{-1}\bm r-\nabla w         =0,\quad
    -\nabla\cdot\bm r-\omega^2\eta w & =-(\Pi_{V_h}u-u)\quad  \text{in } \Omega.
  \end{align}
  According to \eqref{dual_r_w}, there holds
  \begin{align*}
    -\nabla\cdot(\sigma\nabla w)-\omega^2\eta w & =-(\Pi_{V_h}u-u),
  \end{align*}
  which, combining with the regularity \eqref{preg1} and \eqref{preg3}, leads to
  \zzy{\begin{align}\label{reg_w_1}
      \|w\|_{1+\varepsilon} & \le C\|\Pi_{V_h}u-u\|_0   \quad\text{and}\quad
      \|w\|_1                \le C\|\Pi_{V_h}u-u\|_{-1,D}.
    \end{align}}

  Next, we define the discrete problem of \eqref{dual_r_w}: find $(\bm r_h,w_h)\in \bm P_h\times V_h$ such that
  \begin{align}
    a(\bm r_h,w_h;\bm p_h,v_h)=(\Pi_{V_h}u-u,\nabla\cdot\bm p_h+\omega^2\eta v_h) \quad\forall (\bm p_h,v_h)\in \bm P_h\times V_h.
  \end{align}

  Therefore, it is clear that from \eqref{pi-error} and the integration by parts that
  \begin{align*}
     & (\Pi_{V_h}u-u,\nabla\cdot(\bm{\Phi}_h-\bm\Pi_{\bm P_h}\bm\Phi)) +(\Pi_{V_h}u-u,\omega^2\eta(\Psi_h-\Pi_{V_h}\Psi)) \\
     & \quad=a(\bm r_h,w_h;\bm{\Phi}_h-\bm{\Pi}_{\bm P_h}\bm\Phi,\Psi_h-\Pi_{V_h}\Psi)                                    \\&\quad =a(\bm{\Phi}_h-\bm{\Pi}_{\bm P_h}\bm\Phi,\Psi_h-\Pi_{V_h}\Psi;\bm r_h,w_h) \\
     & \quad=(\sigma^{-1}(\bm{\Pi}_{\bm P_h}\bm \Phi-\bm \Phi ),\bm r_h)
    -(\bm{\Pi}_{\bm P_h}\bm \Phi-\bm \Phi,\nabla w_h)
    -(\nabla(\Pi_{V_h}\Psi-\Psi) ,\bm r_h)                                                                                \\
     & \quad\quad
    -\omega^2(\bm{\Pi}_{\bm P_h}\bm \Phi-\bm \Phi ,\nabla(\eta w_h) )
    +\omega^2(\nabla(\eta(\Pi_{V_h}\Psi-\Psi)),\bm r_h)                                                                   \\
     & \quad\quad+\omega^4(\eta^2(\Pi_{V_h}\Psi-\Psi),w_h)                                                                \\
     & \quad\le C \zzyc{(\|\bm{\Pi}_{\bm P_h}\bm \Phi-\bm \Phi\|_0 + \|\Pi_{V_h}\Psi-\Psi\|_0 + \|\nabla(\Pi_{V_h}\Psi-\Psi)\|_0)}(\|\bm r_h\|_0+\|\nabla w_h\|_0).
  \end{align*}
  We apply the triangle inequality and the regularity \eqref{reg_w_1} to find
  \begin{align*}
    \|\bm r_h\|_0+\|\nabla w_h\|_0
     & \le \|\bm r_h-\bm r\|_0+\|\nabla (w_h-w)\|_0+\|\bm r\|_0+\|\nabla w\|_0                                  \\
     & \le Ch^{\varepsilon}(\|\bm r\|_{\varepsilon}+\|w\|_{1+\varepsilon})+\|\sigma \nabla w\|_0+\|\nabla w\|_0 \\
     & \le Ch^{\varepsilon}\|\Pi_{V_h}u-u\|_{0}+C\|\Pi_{V_h}u-u\|_{-1,D}.
  \end{align*}
  Using the Cauchy-Schwarz inequality yields
  \zzyc{\begin{align*}
    |(\Pi_{V_h}u-u,\omega^2\eta(\Psi_h-\Pi_{V_h}\Psi))| & \le C\omega^2\|\Pi_{V_h}u-u\|_{-1,D}\|\nabla(\Psi_h-\Pi_{V_h}\Psi)\|_0,
  \end{align*}}
  \zzyc{which, combing all above estimates directly imply 
  \begin{align*}
    & |(\Pi_{V_h}u-u,\nabla\cdot\bm \Phi_h)|\\
    &\quad\leq (\|\bm{\Pi}_{\bm P_h}\bm \Phi-\bm \Phi\|_0 + \|\Pi_{V_h}\Psi-\Psi\|_0 + \|\nabla(\Pi_{V_h}\Psi-\Psi)\|_0)(h^{\varepsilon}\|\Pi_{V_h}u-u\|_{0}+C\|\Pi_{V_h}u-u\|_{-1,D})\\
    &\quad\quad+Ch\|\nabla\cdot\bm\Phi\|_1\|\Pi_{V_h}u-u\|_0  +C\|\Pi_{V_h}u-u\|_{-1,D}(\omega^2\|\nabla\Psi\|_0+\|\nabla z\|_0).
  \end{align*}
  Then, \eqref{estp-key} and \eqref{estp-key2} follow from \eqref{es11} and \eqref{es10}. The proof is complete.}
\end{proof}

The main theorem comes from \Cref{lem_4.1,lemma_a1,lem4.4}.

\begin{theorem}[The $L^2$ error \zzy{estimate} for $u$]\label{the_u}{For $(\bm q,u)\in \bm H_N^{s}({\rm div};\Omega)\times H^1_D(\Omega)$ with $s>1/2$ and $(\bm q_h,u_h)\in \bm P_h\times V_h$ be the solution of \eqref{lsorg} and \eqref{lsfem}, respectively, there hold}
  \begin{align}\label{l2-err-u-1}
    \begin{split}
      \|\Pi_{V_h}u-u_h\|_0 & \le  C\left(
      h^{\varepsilon} \|\sigma^{-\frac 1 2}(\bm{\Pi}_{\bm P_h}\bm q-\bm q)\|_0
      +\|\Pi_{V_h}u-u\|_0
      \right)                             \\
                           & \quad
      +	C h^{1+\min(\ell,\varepsilon)}\|{\Pi}_{\ell}^o\nabla\cdot\bm q-\nabla\cdot\bm q\|_0,
    \end{split}
  \end{align}
  and when $k\ge 1$,
  \begin{align}\label{l2-err-u-2}
    \begin{split}
      \|\Pi_{V_h}u-u_h\|_0
       & \le
      Ch^{\min(\beta+\varepsilon-1,2\varepsilon)}
      (\|\sigma^{-\frac 1 2}(\bm{\Pi}_{\bm P_h}\bm q-\bm q)\|_0
      +\|\Pi_{\ell}^o\nabla\cdot\bm q-\nabla\cdot\bm q\|_0)
      \\
       & \quad +C(h\|\Pi_{V_h}u-u\|_{0}+\|\Pi_{V_h}u-u\|_{-1,D})                                  \\
       & \quad +Ch^{1+\min(\ell,\varepsilon)}\|\Pi_{\ell}^o\nabla\cdot\bm q-\nabla\cdot\bm q\|_0,
    \end{split}
  \end{align}
  \zzy{whrere $\ell$ is stated in \Cref{lemma-commu} and $\varepsilon,\beta$ are stated in \Cref{ell_pro}.}
\end{theorem}

\begin{proof}
  \zzy{From the stability \eqref{sta} and \eqref{preg1} we have
    \begin{align*}
      \|(\bm\Phi_h,\Psi_h)\|\le C(\|\nabla z\|_0+\|z\|_0)\le C\|\Pi_{V_h}u-u_h\|_0,
    \end{align*}
    which together with \eqref{es_a_1} and the regularity \eqref{dua_u_12} implies that
    \begin{align*}
       & \|\Pi_{V_h}u-u_h\|_0^2 =( \sigma\nabla (\Pi_{V_h} u-u_h),\nabla z  )
      -\omega^2(\eta(\Pi_{V_h} u-u_h) , z)                                                                                                                                                     \\
       & \quad =( \sigma(\nabla (\Pi_{V_h}u-u_h)-\sigma^{-1}(\bm{\Pi}_{\bm P_h}\bm q-\bm q_h) ),\nabla z  )                                                                                    \\
       & \quad\quad-
      (\nabla\cdot(\bm{\Pi}_{\bm P_h} \bm q-\bm q_h)
      +\omega^2\eta(\Pi_{V_h} u-u_h), z)                                                                                                                                                       \\
       & \quad =a(\bm\Phi_h,\Psi_h;\bm{\Pi}_{\bm P_h}\bm q-\bm q_h,\Pi_{V_h}u-u_h)                                                                                                             \\
       & \quad =a(\bm\Phi_h,\Psi_h;\bm{\Pi}_{\bm P_h}\bm q-\bm q,\Pi_{V_h}u-u)                                                                                                                 \\
       & \quad\le Ch^{\varepsilon}(\|\bm\Phi\|_{\varepsilon}+\|\Psi\|_{1+\varepsilon})\|\sigma^{-\frac 1 2}(\bm{\Pi}_{\bm P_h}\bm q-\bm q)\|_0+ C\|(\bm\Phi_h,\Psi_h)\|\cdot\|\Pi_{V_h}u-u\|_0 \\
       & \quad\quad
      + Ch^{1+\min(\ell,\varepsilon)}(\omega^2	\|\Psi\|_{1+\min(\ell,\varepsilon)}+\|z\|_{1+\min(\ell,\varepsilon)})	\|{\Pi}_{\ell}^o\nabla\cdot\bm q-\nabla\cdot\bm q\|_0                     \\
       & \quad\le C\left(h^{\varepsilon}\|\sigma^{-\frac 1 2}(\bm{\Pi}_{\bm P_h}\bm q-\bm q)\|_0
      +\|\Pi_{V_h}u-u\|_0\right)\|\Pi_{V_h}u-u_h\|_0                                                                                                                                           \\
       & \quad\quad+	C h^{1+\min(\ell,\varepsilon)}	\|\Pi_{V_h}u-u_h\|_0\|{\Pi}_{\ell}^o\nabla\cdot\bm q-\nabla\cdot\bm q\|_0.
    \end{align*}}

  \zzy{By \eqref{es12} and the regularity \eqref{dua_u_12}, there holds
    \begin{align*}
      \|\nabla \Psi_h\|_0\le\|\nabla (\Psi_h - \Psi)\|_0+\|\nabla \Psi\|_0\le C\|\Pi_{V_h}u-u_h\|_0,
    \end{align*}
    together with \eqref{E-error3}, the regularity \eqref{dua_u_12} and \eqref{dua_u_13}, lead to
    \begin{align*}
       & \|\Pi_{V_h}u-u_h\|_0^2 =a(\bm\Phi_h,\Psi_h;\bm{\Pi}_{\bm P_h}\bm q-\bm q,\Pi_{V_h}u-u) \\
       & \quad \le
      Ch^{\min(\beta+\varepsilon-1,2\varepsilon)}(\|\bm\Phi\|_{\beta}+\|\Psi\|_{1+\varepsilon})
      (\|\sigma^{-\frac 1 2}(\bm{\Pi}_{\bm P_h}\bm q-\bm q)\|_0                                 \\
       & \quad\quad+\|\Pi_{\ell}^o\nabla\cdot\bm q-\nabla\cdot\bm q\|_0)
      \\
       & \quad \quad+Ch^{\varepsilon}(
      \|\bm\Phi\|_{\varepsilon}+
      \|\Psi\|_{1+\varepsilon}
      )(h^{\varepsilon}\|\Pi_{V_h}u-u\|_{0}+\|\Pi_{V_h}u-u\|_{-1,D})
      \\
       & \quad \quad
      +Ch\|\Pi_{V_h}u-u\|_0
      \|\nabla\cdot\bm \Phi\|_1
      +
      C\|\Pi_{V_h}u-u\|_{-1,D}(\omega^2\|\nabla\Psi\|_0+\|\nabla z\|_0)                         \\
       & \quad \quad
      +C \omega^4\|\nabla\Psi_h\|_0\|\Pi_{V_h}u-u\|_{-1,D}
      \\
       & \quad \quad
      +C	h^{1+\min(\ell,\varepsilon)}\|\Pi_{\ell}^o\nabla\cdot\bm q-\nabla\cdot\bm q\|_0
      (
      \omega^2\|\Psi\|_{1+\min(\ell,\varepsilon)}
      +	\|z\|_{1+\min(\ell,\varepsilon)}
      )                                                                                         \\
       & \quad \le
      Ch^{\min(\beta+\varepsilon-1,2\varepsilon)}
      (\|\sigma^{-\frac 1 2}(\bm{\Pi}_{\bm P_h}\bm q-\bm q)\|_0
      +\|\Pi_{\ell}^o\nabla\cdot\bm q-\nabla\cdot\bm q\|_0)                                     \\
       & \quad\quad\cdot\|\Pi_{V_h}u-u_h\|_0
      +C \left(h\|\Pi_{V_h}u-u\|_{0}+\|\Pi_{V_h}u-u\|_{-1,D}\right)\|\Pi_{V_h}u-u_h\|_0         \\
       & \quad \quad
      +Ch^{1+\min(\ell,\varepsilon)}\|\Pi_{\ell}^o\nabla\cdot\bm q-\nabla\cdot\bm q\|_0\|\Pi_{V_h}u-u_h\|_0.
    \end{align*}
    This completes the proof.}
\end{proof}

Some direct results can be arrived in \eqref{l2-err-u-1} and \eqref{l2-err-u-2}. In the {following} remarks, $(\bm q,u)$ is the solution of \eqref{lsorg} which is assumed to be smooth enough.
\begin{remark}[Suboptimal estimate]
  For the pair $\bm{\mathcal{BDM}}_{1}/\mathcal P_{2}$, only the suboptimal convergence estimate holds
  \begin{align*}
    \|\Pi_{V_h}u-u_h\|_0 \le h^2(\|q\|_2+\|\nabla \cdot q\|_1+\|u\|_3).
  \end{align*}
\end{remark}
\begin{remark}[Optimal estimates]
  For the case $\Omega$ satisfies $\varepsilon=1$, besides the results presented in \Cref{rem_3.13}, more optimal convergence estimates can be obtained from \eqref{l2-err-u-1} that
  \zzy{\begin{align*}
       & \|\Pi_{V_h}u-u_h\|_0  \le Ch^{m+1}(\|\bm q\|_{m}+\|\nabla \cdot \bm q\|_{m-1}+ \|u\|_{m+1})
    \end{align*}
    for $\bm{\mathcal{BDM}}_{m-1}/\mathcal P_{m}, m>2$; $\bm{\mathcal{BDM}}_{m}/\mathcal P_{m}, m>1$; $\bm{\mathcal{BDM}}_{m+1}/\mathcal P_{m}, m \ge 1$; $\bm{\mathcal{RT}}_{k}/\mathcal P_{m},$ $ k=m-1,m,m+1$.}
  \zzy{\zzyb{Compared} to \Cref{rem_3.12}, these optimal estimates require only $\bm q\in \bm H^m$ while $u\in H^{m+1}$.}
\end{remark}
\begin{remark}[\zzya{Supercloseness} estimates]
  For the domain satisfying $\beta = 2$ and $\varepsilon = 1$, from \eqref{l2-err-u-2}, the \zzya{supercloseness} estimates hold when $m>1$ that
  \zzy{\begin{align*}
       & \|\Pi_{V_h}u-u_h\|_0  \le Ch^{m+2} (\|\bm q\|_{m}+\|\nabla \cdot \bm q\|_{m}+ \|u\|_{m+1})
    \end{align*}
    for $\bm{\mathcal{BDM}}_{k}/\mathcal P_{m},k=m,m+1$; $\bm{\mathcal{RT}}_{k}/\mathcal P_{m},k=m-1,m,m+1$.}
\end{remark}

\subsection{Error \zzy{estimates} for $\nabla\cdot(\mathbf\Pi_{\mathbf P_h}\mathbf q-\mathbf q_h)$}\label{dua_div_q}

In this section, \zzyb{we} perform the dual argument for $\nabla\cdot(\bm\Pi_{\bm P_h}\bm q-\bm q_h)$ to get the \zzya{supercloseness} result for some specific choices of $\bm P_h$ and $V_h$.

We define the problem: find $(\bm\Phi,\Psi)\in \bm H_N({\rm div};\Omega)\times H^1_D(\Omega)$ such that
\begin{align*}
  \sigma^{-1}\bm{\Phi}-\nabla \Psi        =\bm 0                                       , \quad
  -\nabla\cdot\bm \Phi-\omega^2\eta \Psi  =-\nabla\cdot(\bm\Pi_{\bm P_h}\bm q-\bm q_h) \quad \text{in }\Omega.
\end{align*}
Then the regularity result holds
\begin{align}\label{reg4}
  \|\bm\Phi\|_{\varepsilon}
  +\|\Psi\|_{1+\varepsilon}\le C\|\nabla\cdot(\bm\Pi_{\bm P_h}\bm q-\bm q_h)\|_0.
\end{align}
We define the discrete problem: find $(\bm\Phi_h,\Psi_h)\in \bm P_h\times V_h$ such that
\begin{align}\label{div-dual-dis}
  a(\bm\Phi_h,\Psi_h;\bm p_h,v_h)=(\nabla\cdot(\bm\Pi_{\bm P_h}\bm q-\bm q_h),\nabla\cdot\bm p_h+\omega^2 \eta v_h)
\end{align}
holds for any $(\bm p_h,v_h)\in \bm P_h\times V_h$. Then we can prove the following result.
\begin{theorem}[The $L^2$ error estimate for $\nabla \cdot \bm q$]\label{the_div_q} {For $(\bm q,u)\in \bm H_N^{s}({\rm div};\Omega)\times H^1_D(\Omega)$ with $s>1/2$ and $(\bm q_h,u_h)\in \bm P_h\times V_h$ be the solution of \eqref{lsorg} and \eqref{lsfem}, respectively, there holds}
  \begin{align}\label{div_q}
    \begin{split}
      \|\nabla\cdot(\bm\Pi_{\bm P_h}\bm q-\bm q_h)\|_0
       & \le C\left(h^{\varepsilon} \|\sigma^{-\frac 1 2}(\bm{\Pi}_{\bm P_h}\bm q-\bm q)\|_0
      +\|\Pi_{V_h}u-u\|_0\right)                                                                               \\
       & \quad+ C\omega^2 h^{1+\min(\ell,\varepsilon)}\|{\Pi}_{\ell}^o(\nabla\cdot\bm q)-\nabla\cdot\bm q\|_0,
    \end{split}
  \end{align}
  \zzy{whrere $\ell$ is stated in \Cref{lemma-commu} and $\varepsilon$ is stated in \Cref{ell_pro}.}
\end{theorem}
\zzy{ \begin{proof} We take $\bm p_h=\bm\Pi_{\bm P_h}\bm q-\bm q_h\in \bm P_h$ in \eqref{div-dual-dis}, together with the stability \eqref{sta}, \eqref{es_a_1} and the regularity \eqref{reg4} to get
    \begin{align*}
       & \|\nabla\cdot(\bm\Pi_{\bm P_h}\bm q-\bm q_h)\|_0^2+(\nabla\cdot(\bm\Pi_{\bm P_h}\bm q-\bm q_h),\omega^2 \eta (\Pi_{V_h}u-u_h))                                                         \\
       & \quad =a(\bm\Phi_h,\Psi_h;\bm\Pi_{\bm P_h}\bm q-\bm q_h,\Pi_{V_h}u-u_h)=a(\bm\Phi_h,\Psi_h;
      \bm\Pi_{\bm P_h}\bm q-\bm q,\Pi_{V_h}u-u)                                                                                                                                                 \\
       & \quad \le Ch^{\varepsilon}(\|\bm\Phi\|_{\varepsilon}+\|\Psi\|_{1+\varepsilon})\|\sigma^{-\frac 1 2}(\bm{\Pi}_{\bm P_h}\bm q-\bm q)\|_0+ C\|(\bm\Phi_h,\Psi_h)\|\cdot\|\Pi_{V_h}u-u\|_0 \\
       & \quad\quad+ Ch^{1+\min(\ell,\varepsilon)}\omega^2\|\Psi\|_{1+\min(\ell,\varepsilon)}\|{\Pi}_{\ell}^o\nabla\cdot\bm q-\nabla\cdot\bm q\|_0                                              \\
       & \quad \le C\left( h^{\varepsilon} \|\sigma^{-\frac 1 2}(\bm{\Pi}_{\bm P_h}\bm q-\bm q)\|_0 +\|\Pi_{V_h}u-u\|_0 \right)	\|\nabla\cdot(\bm\Pi_{\bm P_h}\bm q-\bm q_h)\|_0                \\
       & \quad\quad+ Ch^{1+\min(\ell,\varepsilon)}\omega^2\|\nabla\cdot(\bm\Pi_{\bm P_h}\bm q-\bm q_h)\|_0\|{\Pi}_{\ell}^o(\nabla\cdot\bm q)-\nabla\cdot\bm q\|_0.
    \end{align*}
    The estimate \eqref{div_q} is established by the above result, the triangle inequality, the Cauchy-Schwarz inequality and the estimate of $\|\Pi_{V_h}u-u_h\|_0$ from \eqref{l2-err-u-1}. This completes the proof.
  \end{proof}
}

\begin{remark}[{Optimal and \zzya{Supercloseness} estimates}]\label{rem_su_div_q}
  For $\Omega$ satisfies $\varepsilon=1$, for $\bm{\mathcal{BDM}}_{k}$ (or $\bm{\mathcal{RT}}_{k}$)$/\mathcal P_{k+1}$, a higher order of convergence rate can be obtained than that presented in \Cref{rem_3.13}
  \zzy{\begin{align*}
       & \|\nabla\cdot(\bm\Pi_{\bm P_h}\bm q-\bm q_h)\|_0 \\
       & \quad\le Ch^{k+2}\left\{
      \begin{aligned}
                                                  & ( \|\bm q\|_{k+1}
        +\|\nabla\cdot\bm q\|_{k}+\|u\|_{k+2}),   & \bm{\mathcal{BDM}}_{k}/\mathcal P_{k+1}(k>1), \\
                                                  & ( \|\bm q\|_{k+1}
        +\|\nabla\cdot\bm q\|_{k}+\|u\|_{k+2}),   & \bm{\mathcal{RT}}_{k}/\mathcal P_{k+1}(k>0),  \\
                                                  & ( \|\bm q\|_{k+1}
        +\|\nabla\cdot\bm q\|_{k+1}+\|u\|_{k+2}), & \bm{\mathcal{RT}}_{k}/\mathcal P_{k+1}(k=0).
      \end{aligned}
      \right.
    \end{align*}}
  {For $\bm{\mathcal{BDM}}_{k}/\mathcal P_{k+1}$, the estimate is \zzya{supercloseness} and for $\bm{\mathcal{RT}}_{k}/\mathcal P_{k+1}$, it is optimal.} When we apply the element $\bm{\mathcal{BDM}}_{1}/\mathcal P_{2}$, the different order of \zzya{supercloseness} rates can be obtained with and without the $\omega$:
  \begin{align*}
    \|\nabla\cdot(\bm\Pi_{\bm P_h}\bm q-\bm q_h)\|_0\le C\left\{
    \begin{aligned}
                                            & h^{2} ( \|\bm q\|_{2}
      +\|\nabla\cdot\bm q\|_{1}+\|u\|_{3}), & \omega \neq 0,                                  \\
                                            & h^{3} ( \|\bm q\|_{2}+\|u\|_{3}), & \omega = 0.
    \end{aligned}
    \right.
  \end{align*}
\end{remark}

\subsection{Error \zzy{estimates} for $\mathbf\Pi_{\mathbf P_h}\mathbf q-\mathbf q_h$}\label{dua_q}

Let us first introduce some regularity results which will be used in the dual argument of $\bm\Pi_{\bm P_h}\bm q-\bm q_h$.
\begin{lemma}For any given $\bm\theta\in \bm H({\rm div};\Omega)$, the problem reads:
  find $(\bm t,z)\in \bm H_N({\rm div};\Omega)\times H^1_D(\Omega)$ such that
  \begin{align}\label{dualp1}
    \sigma^{-1}\bm t-\nabla z         =\sigma^{-1}\bm \theta,\quad
    -\nabla\cdot\bm t-\omega^2\eta z  =0    \quad                  \text{in }\Omega.
  \end{align}
  The second equation of \eqref{dualp1} is in the distribution sense between space $H_D^1$ and space $H_D^{-1}$.
  Then there holds the regularity estimates
  \begin{align}\label{reg-dualp1}
    \zzyc{\|\bm t\|_{0} +	\|z\|_{1}\le C\|\bm\theta\|_0,\qquad
    \|z\|_{1+\varepsilon}\le C\|\nabla\cdot\bm\theta\|_0}.
  \end{align}
\end{lemma}
\begin{proof}
  A direct calculation shows that
  \begin{align*}
    -\nabla\cdot(\sigma\nabla z)-\omega^2\eta z =\nabla\cdot\bm{\theta}.
  \end{align*}
  By the regularity assumptions \eqref{preg3}, \eqref{preg1} and the Poincaré inequality, we obtain
  \zzy{\begin{align*}
      \|z\|_1\le 	C\|\nabla z\|_0\le C\|\nabla \cdot \bm\theta\|_{-1,D}\le C\|\bm\theta\|_0,\
      \|\nabla z\|_{\varepsilon}\le C\|\nabla \cdot \bm \theta\|_{-1+\varepsilon}\le C\|\bm \theta\|_{\varepsilon},
    \end{align*}}
  which, together with \eqref{dualp1} and the triangle inequality, bring us that
  \zzyc{\begin{align*}
    \|\bm t\|_0=\|\sigma\nabla z+\bm\theta\|_0\le C\|\bm\theta\|_0.
  \end{align*}}
  Also, from regularity assumption \eqref{preg1}, it is reached that $\|z\|_{1+\varepsilon}\le C\|\nabla \cdot \bm\theta\|_0$.
  Therefore, the proof is finished.
\end{proof}

\begin{lemma}Let $(\bm t,z)$ be the solution of \eqref{dualp1}, we define the problem:
  find $(\bm \Phi,\Psi)\in \bm H_N({\rm div};\Omega)\times H^1_D(\Omega)$ such that
  \begin{align}\label{dualp2}
    \sigma^{-1}\bm\Phi-\nabla\Psi          =\sigma^{-1}\bm t, \quad
    -\nabla\cdot\bm\Phi-\omega^2\eta \Psi  =-z \quad  \text{in }\Omega.
  \end{align}
  Then the estimates hold
  \begin{align}\label{reg-dualp2}
    \|\Psi\|_{1+\varepsilon}+\|\nabla\cdot\bm\Phi\|_{1}\le C\|\bm\theta\|_0,\qquad
    \zzyc{\|\bm\Phi\|_{0}\le C\|\bm\theta\|_{0}}.
  \end{align}
\end{lemma}
\begin{proof}
  By a direct calculation, it holds
  \begin{align*}
    \nabla\cdot(\sigma\nabla\Psi)+\omega^2\eta \Psi=(\omega^2\eta + 1)z,
  \end{align*}
  which, together with the regularity \eqref{preg1}, leads to $\|\Psi\|_{1+\varepsilon}\le C\|z\|_0\le C\|\bm\theta\|_0$.
  Combining with the regularity \eqref{reg-dualp1} and \eqref{dualp2}, we obtain
  \begin{align*}
    \zzyc{\|\bm\Phi\|_{0}\le C(\|t\|_{0}+\|\Psi\|_{1}) \le C\|\bm\theta\|_{0}} \quad
    \text{and}\quad
    \|\nabla\cdot\bm\Phi\|_{1}\le C\|\bm\theta\|_0,
  \end{align*}
  which finishes the proof.
\end{proof}

We define the discrete problem the corresponding to \eqref{dualp2}: find $(\bm \Phi_h,\Psi_h)\in \bm P_h\times V_h$ such that
\begin{align}\label{dualp3-dis}
  a(\bm \Phi_h,\Psi_h;\bm p_h,v_h)=(t,\sigma^{-1}\bm p_h-\nabla v_h)+(z,\nabla\cdot\bm p_h+\omega^2\eta v_h)
\end{align}
for any $(\bm p_h,v_h)\in \bm P_h\times V_h$.

\begin{theorem}[The $L^2$ error \zzy{estimate} for $\bm q$]\label{the_q}\zzyc{For $(\bm q,u)\in \bm H_N^{s}({\rm div};\Omega)\times H^1_D(\Omega)$ with $s>1/2$ and $(\bm q_h,u_h)\in \bm P_h\times V_h$ be the solution of \eqref{lsorg} and \eqref{lsfem},} then the following estimate holds
  \begin{align}\label{sig_q}
    \begin{split}
       & \|\sigma^{-\frac 1 2}(\bm\Pi_{\bm P_h}\bm q-\bm q_h)\|_0 \\
       & \qquad  \le C
      \left(\|\sigma^{-\frac 1 2}(\bm\Pi_{\bm P_h}\bm q-\bm q)\|_0
      +h^{\varepsilon}\|\Pi_{V_h}u-u\|_{0}
      +\|\Pi_{V_h}u-u\|_{-1,D}\right)
    \end{split}
  \end{align}
  \zzy{whrere $\varepsilon$ is stated in \Cref{ell_pro}.}
\end{theorem}
\begin{proof}
    We take $\bm\theta=\bm{\Pi}_{\bm P_h}\bm q-\bm q_h$ in \eqref{dualp1}, and use the \eqref{dualp1}, \eqref{dualp2}, the orthogonality \eqref{ell_p_1}, the integration by parts and \eqref{dualp3-dis} to get
    \begin{align*}
       & \|\sigma^{-\frac 1 2}(\bm{\Pi}_{\bm P_h}\bm q-\bm q_h)\|_0^2                                                                                                    \\
       & \quad=(\sigma^{-1}\bm t-\nabla z,\bm{\Pi}_{\bm P_h}\bm q-\bm q_h)=(\sigma^{-1}\bm t,\bm{\Pi}_{\bm P_h}\bm q-\bm q_h)-(\nabla z,\bm{\Pi}_{\bm P_h}\bm q-\bm q_h) \\
       & \quad =(\sigma^{-1}\bm t,\bm{\Pi}_{\bm P_h}\bm q-\bm q_h-\sigma\nabla(\Pi_{V_h} u- u_h))                                                                        \\
       & \quad\quad+( z,\nabla\cdot(\bm{\Pi}_{\bm P_h}\bm q-\bm q_h)+\omega^2\eta(\Pi_{V_h} u- u_h ))                                                                    \\
       & \quad=a(\bm\Phi_h,\Psi_h;\bm\Pi_{\bm P_h}\bm q-\bm q_h,\Pi_{V_h}u-u_h)=a(\bm\Phi_h,\Psi_h;\bm\Pi_{\bm P_h}\bm q-\bm q,\Pi_{V_h}u-u).
    \end{align*}
    We recall the stability \eqref{sta} and the regularity \eqref{reg-dualp1} to get
    \begin{align}\label{the6.3_2}
      \|(\bm\Phi_h,\Psi_h)\|\le C(\|\sigma^{-1}\bm t\|_0+\|z\|_0)\le C\|\bm\theta\|_0\le C\|\sigma^{-\frac 1 2}(\bm\Pi_{\bm P_h}\bm q-\bm q_h)\|_0.
    \end{align}
    From the estimate of $(\Pi_{V_h}u-u ,\nabla\cdot\bm \Phi_h)$ in \zzyc{\eqref{estp-key2}}, the Poincaré inequality, \eqref{reg-dualp1}, \eqref{reg-dualp2}, \eqref{the6.3_2} and \eqref{pi-error}, we arrive that
    \begin{align}\label{the6.3_1}
      \begin{split}
         & a(\bm\Phi_h,\Psi_h;\bm\Pi_{\bm P_h}\bm q-\bm q,\Pi_{V_h}u-u) =a(\bm{\Pi}_{\bm P_h}\bm q-\bm q,\Pi_{V_h}u-u;\bm\Phi_h,\Psi_h) \\
         & \quad =(\sigma^{-1}(\bm{\Pi}_{\bm P_h}\bm q-\bm q ),\bm \Phi_h)
        -(\bm{\Pi}_{\bm P_h}\bm q-\bm q,\nabla \Psi_h)
        +(\Pi_{V_h}u-u ,\nabla\cdot\bm \Phi_h)                                                                                          \\
         & \quad\quad
        -\omega^2(\bm{\Pi}_{\bm P_h}\bm q-\bm q ,\nabla(\eta \Psi_h) )
        +\omega^2(\eta(\Pi_{V_h}u-u),\nabla\cdot\bm \Phi_h)                                                                             \\
         & \quad\quad+\omega^4(\eta^2(\Pi_{V_h}u-u),\Psi_h)                                                                             \\
         & \quad\le C
        \|\sigma^{-\frac 1 2}(\bm\Pi_{\bm P_h}\bm q-\bm q)\|_0
        \|(\bm \Phi_h, \Psi_h)\|+C\omega^4\|\Pi_{V_h}u-u\|_{-1,D}\|\nabla \Psi_h\|_0
        \\
         & \quad\quad +C|(\Pi_{V_h}u-u,\nabla\cdot\bm \Phi_h)|                                                                          \\
         & \quad\le  C
        (\|\sigma^{-\frac 1 2}(\bm\Pi_{\bm P_h}\bm q-\bm q)\|_0+\|\Pi_{V_h}u-u\|_{-1,D})
        \|\sigma^{-\frac 1 2}(\bm\Pi_{\bm P_h}\bm q-\zzyc{\bm q_h})\|_0                                                                          \\
         & \quad\quad+ \zzyc{C 
        \|\sigma^{-\frac 1 2}(\bm\Pi_{\bm P_h}\bm q-\bm q_h)\|_{0}}
        (h^{\varepsilon}\|\Pi_{V_h}u-u\|_{0}+\|\Pi_{V_h}u-u\|_{-1,D})
        \\
         & \quad \quad
        +Ch\|\Pi_{V_h}u-u\|_0\|\sigma^{-\frac 1 2}(\bm\Pi_{\bm P_h}\bm q-\zzyc{\bm q_h})\|_0.
      \end{split}
    \end{align}

    The estimate \eqref{sig_q} follows from \eqref{the6.3_1} and \eqref{the6.3_2}.
  \end{proof}

\begin{remark}[Suboptimal convergence estimate]
  For $\Omega$ satisfies $\beta=1$ and $\varepsilon=1$, and for $\bm{\mathcal{BDM}}_{2}/\mathcal P_{1}$ and $\bm{\mathcal{RT}}_{2}/\mathcal P_{1}$, it holds
  \begin{align*}
    \|\sigma^{-\frac 1 2}(\bm\Pi_{\bm P_h}\bm q-\bm q_h)\|_0\le Ch^{2}
    (\|\bm q\|_{\zzy{2}}+ \|u\|_{2}).
  \end{align*}
\end{remark}
\begin{remark}[Optimal convergence estimates]For $\Omega$ satisfies $\beta=1$ and $\varepsilon=1$, and for the element $\bm{\mathcal{RT}}_{0}/\mathcal P_{1}$, it holds
  \begin{align*}
    \|\sigma^{-\frac 1 2}(\bm\Pi_{\bm P_h}\bm q-\bm q_h)\|_0\le C h(\|\bm q\|_1+\|u\|_{\zzy{1}}),
  \end{align*}
  and for the elements $\bm{\mathcal{BDM}}_{k}(\text{or }\bm{\mathcal{RT}}_{k})/\mathcal P_{k+1}$, the optimal convergence estimates hold
  \begin{align*}
    \|\sigma^{-\frac 1 2}(\bm\Pi_{\bm P_h}\bm q-\bm q_h)\|_0\le C h^{k+1}(\|\bm q\|_{k+1}+\|u\|_{\zzy{k+1}}).
  \end{align*}
  As shown in \Cref{rem_3.5}, for $\Omega$ satisfies $\beta=2$ and $\varepsilon=1$, and for the elements $\bm{\mathcal{BDM}}_{k}/\mathcal P_{k-1}$ and $\bm{\mathcal{RT}}_{k}/\mathcal P_{k-1}$ with $k > 2$, it holds
  \begin{align*}
    \|\sigma^{-\frac 1 2}(\bm\Pi_{\bm P_h}\bm q-\bm q_h)\|_0\le C h^{k+1}(\|\bm q\|_{k+1}+\|u\|_{k}).
  \end{align*}
\end{remark}

\subsection{Error \zzy{estimates} for $\nabla(\Pi_{V_h}u-u_h)$}\label{dua_grad_u}

For the above $z\in H^1_D(\Omega)$, we define the problem: find $(\bm \Phi,\Psi)\in \bm H_N({\rm div};\Omega)\times H^1_D(\Omega)$ such that
\begin{align}\label{dualgu2}
  \sigma^{-1}\bm\Phi-\nabla\Psi          =-\nabla z, \quad
  -\nabla\cdot\bm\Phi-\omega^2\eta \Psi  =z     \quad      \text{in }\Omega.
\end{align}

We define the discrete problem the corresponding to \eqref{dualgu2}: find $(\bm \Phi_h,\Psi_h)\in \bm P_h\times V_h$ such that
\begin{align}
  a(\bm \Phi_h,\Psi_h;\bm p_h,v_h)=(-\sigma\nabla z,\sigma^{-1}\bm p_h-\nabla v_h)-(z,\nabla\cdot\bm p_h+\omega^2\eta v_h)\label{dual_grad_u_1}
\end{align}
holds for all $(\bm p_h,v_h)\in \bm P_h\times V_h$.
Then it holds
\begin{align*}
  a(\bm \Phi_h,\Psi_h;\bm p_h,v_h)=(\sigma\nabla z,\nabla v_h)-(z,\omega^2\eta v_h).
\end{align*}

Next, we can obtain the following regularity \zzy{estimates}.
\begin{lemma} Let $(\bm \Phi,\Psi)\in \bm H_N({\rm div};\Omega)\times H^1_D(\Omega)$ be the solution of  \eqref{dualgu2}, then
  \begin{align}
    \|\bm \Phi\|_{\varepsilon}\le C\|z\|_0,   \qquad   \|\Psi\|_{1}\le C\|\nabla z\|_0.\label{7reg2}
  \end{align}
\end{lemma}
\begin{proof}
  A direct calculation shows that
  \begin{align*}
    -\nabla\cdot(\sigma\nabla (\Psi-z))-\omega^2\eta(\Psi-z)  =\omega^2\eta z+z.
  \end{align*}
  Therefore, using the regularity assumption \eqref{preg1} and Poincaré inequality yields
  \begin{align*}
    \|\Psi-z\|_{1+\varepsilon}\le C\|z\|_0,\qquad	\|\Psi-z\|_1\le C\|z\|_0\le C\|\nabla z\|_0.
  \end{align*}
  According to \eqref{dualgu2}, there holds
  \begin{align*}
    \|\Phi\|_{\varepsilon}=\|\sigma(\nabla (\Psi - z))\|_{\varepsilon}\le C\|\Psi-z\|_{1+\varepsilon},
  \end{align*}
  which leads to \eqref{7reg2}.
\end{proof}

For estimating $\|\sigma^{\frac 1 2}\nabla(\Pi_{V_h}u-u_h)\|_0$, the required regularity is different from all previous estimates. \zzy{It should be noted} that for this term, only $\Psi\in H^{1}(\Omega)$ holds in \eqref{7reg2}. For $|(\Pi_{V_h}u-u,\nabla\cdot\bm\Phi_h)|$, \zzy{a different estimate is derived.}

\begin{lemma}
  For the above discrete problem \eqref{dual_grad_u_1}, the following estimate holds
  \begin{align}
    \begin{split}\label{estu-key}
       & |(\Pi_{V_h}u-u,\nabla\cdot\bm\Phi_h)|\le
      Ch^{\varepsilon}\|\Pi_{V_h}u-u\|_{0}(
      \|\bm\Phi\|_{\varepsilon}+
      \|\Psi\|_{1}
      )
      \\
       & \qquad
      +Ch\|\Pi_{V_h}u-u\|_0
      \|\nabla\cdot\bm \Phi\|_1
      +
      C\|\Pi_{V_h}u-u\|_{-1,D}(\omega^2\|\nabla\Psi\|_0+\|\nabla z\|_0).
    \end{split}
  \end{align}
\end{lemma}
\begin{proof}
  We find from the estimate \eqref{es11} that
  \begin{align*}
    \|\nabla\cdot(\bm{\Phi}_h-\bm\Pi_{\bm P_h}\bm\Phi)\|_0\le Ch^{\varepsilon}(\|\bm\Phi\|_{\varepsilon}+\|\Psi\|_1).
  \end{align*}
  Based on \eqref{dualgu2}, \eqref{neg_norm}, the triangle inequality and the above estimate, we conclude that
  \begin{align*}
     & |(\Pi_{V_h}u-u,\nabla\cdot\bm \Phi_h)|
    =|(\Pi_{V_h}u-u,\nabla\cdot(\bm{\Phi}_h-\bm\Phi)-\omega^2\eta\Psi-z)|                    \\
     & \qquad \le |(\Pi_{V_h}u-u,\nabla\cdot(\bm{\Phi}_h-\bm\Phi))|
    +|(\Pi_{V_h}u-u,\omega^2\eta \Psi + z)|                                                  \\
     & \qquad \le Ch^{\varepsilon}\|\Pi_{V_h}u-u\|_0(\|\bm\Phi\|_{\varepsilon}+\|\Psi\|_{1})
    +Ch\|\nabla\cdot\bm\Phi\|_1\|\Pi_{V_h}u-u\|_0                                            \\
     & \qquad \quad+
    C\|\Pi_{V_h}u-u\|_{-1,D}(\omega^2\|\nabla\Psi\|_0+\|\nabla z\|_0),
  \end{align*}
  which completes our proof.
\end{proof}

\begin{theorem}[The $L^2$ error \zzy{estimate} for $\nabla u$]\label{the_grad_u} {For $(\bm q,u)\in \bm H_N^{s}({\rm div};\Omega)$\\$\times H^1_D(\Omega)$ with $s>1/2$ and $(\bm q_h,u_h)\in \bm P_h\times V_h$ be the solution of \eqref{lsorg} and \eqref{lsfem}, respectively, there hold}
    \begin{align}\label{the7.3_1}
      \begin{split}
        \|\sigma^{\frac 1 2}\nabla(\Pi_{V_h}u-u_h)\|_0
         & \le C h^{\varepsilon} \left(
        \|\sigma^{-\frac 1 2}(\bm\Pi_{\bm P_h}\bm q-\bm q)\|_0
        +
        \|\Pi_{\ell}^o(\nabla\cdot\bm q)-\nabla\cdot\bm q\|_0\right)                            \\
         & \quad+C ( (h^{\varepsilon}+\omega^2) \|\Pi_{V_h}u-u\|_{0} +\|\Pi_{V_h}u-u\|_{-1,D}),
      \end{split}
    \end{align}
    and when $k\ge 1$,
  \begin{align}\label{the7.3_2}
    \begin{split}
       & \|\sigma^{\frac 1 2}\nabla(\Pi_{V_h}u-u_h)\|_0 \le C ( h^{\varepsilon}\|\Pi_{V_h}u-u\|_{0}  +\|\Pi_{V_h}u-u\|_{-1,D}) \\
       & \quad\quad+C (h^{\varepsilon}+\omega^2 h^{\beta+\varepsilon-1}) \left(
      \|\sigma^{-\frac 1 2}(\bm\Pi_{\bm P_h}\bm q-\bm q)\|_0
      +
      \|\Pi_{\ell}^o(\nabla\cdot\bm q)-\nabla\cdot\bm q\|_0\right),
    \end{split}
  \end{align}
  \zzy{whrere $\ell$ is stated in \Cref{lemma-commu} and $\varepsilon,\beta$ are stated in \Cref{ell_pro}.}
\end{theorem}
\zzy{\begin{proof}
    We take $z=\Pi_{V_h}u-u_h$ to get
    \begin{align*}
       & \|\sigma^{\frac 1 2}\nabla(\Pi_{V_h}u-u_h)\|_0^2+(\Pi_{V_h}u-u_h,\omega^2\eta (\Pi_{V_h}u-u_h))                                       \\
       & \quad =a(\bm\Phi_h,\Psi_h;\bm\Pi_{\bm P_h}\bm q-\bm q_h,\Pi_{V_h}u-u_h)=a(\bm\Phi_h,\Psi_h;\bm\Pi_{\bm P_h}\bm q-\bm q,\Pi_{V_h}u-u).
    \end{align*}
    It follows from the Cauchy-Schwarz inequality that
    \begin{align*}
      |(\Pi_{V_h}u-u_h,\omega^2\eta (\Pi_{V_h}u-u_h))|\le C\omega^2\|\Pi_{V_h}u-u_h\|^2_0.
    \end{align*}
    Using \eqref{pi-error} yields
    \begin{align*}
       & a(\bm\Phi_h,\Psi_h;\bm\Pi_{\bm P_h}\bm q-\bm q,\Pi_{V_h}u-u)                                                    \\
       & \quad=(\sigma^{-1}(\bm{\Pi}_{\bm P_h}\bm q-\bm q ),\bm\Phi_h)
      -(\bm{\Pi}_{\bm P_h}\bm q-\bm q,\nabla \Psi_h)
      +(\Pi_{V_h}u-u ,\nabla\cdot\bm\Phi_h)                                                                              \\
       & \quad\quad
      -\omega^2(\bm{\Pi}_{\bm P_h}\bm q-\bm q ,\nabla(\eta \Psi_h) )
      +\omega^2(\eta(\Pi_{V_h}u-u),\nabla\cdot\bm\Phi_h)                                                                 \\
       & \quad\quad+\omega^4(\eta^2(\Pi_{V_h}u-u),\Psi_h)                                                                \\
       & \quad=(\sigma^{-1}(\bm{\Pi}_{\bm P_h}\bm q-\bm q ),\bm\Phi_h-\bm\Phi)
      -(\bm{\Pi}_{\bm P_h}\bm q-\bm q,\nabla (\Psi_h-\Psi)+\nabla z)
      \\
       & \quad\quad+(\Pi_{V_h}u-u , \nabla\cdot\bm\Phi_h) -\omega^2(\bm{\Pi}_{\bm P_h}\bm q-\bm q ,\nabla(\eta \Psi_h) ) \\
       & \quad\quad
      +\omega^2(\eta(\Pi_{V_h}u-u),\nabla\cdot\bm\Phi_h)                                                                 +\omega^4(\eta^2(\Pi_{V_h}u-u),\Psi_h)=:\sum_{i=1}^6R_i.
    \end{align*}
    By \eqref{es11}, we obtain \zzy{the estimates} of $R_1$ that
    \begin{align*}
      |R_1|\le Ch^{\varepsilon}(\|\bm\Phi\|_{\varepsilon}+\|\Psi\|_{1})\|\sigma^{-\frac 1 2}(\bm{\Pi}_{\bm P_h}\bm q-\bm q)\|_0.
    \end{align*}
    Also, we apply \eqref{es11}, the triangle inequality and the Poincaré inequality to get
    \begin{align*}
      |R_2| & =|(\bm{\Pi}_{\bm P_h}\bm q-\bm q,\nabla (\Psi_h-\Psi)+\nabla (z-\Pi_0^o z))|                                                                     \\
            & =|(\Pi_{\ell}^o(\nabla\cdot\bm q)-\nabla\cdot\bm q,(\Psi_h-\Pi_{V_h}\Psi)+(\Pi_{V_h}\Psi-\Psi)+(z-\Pi_0^o z))|                                   \\
            & \le C(\|\nabla(\Psi_h-\Pi_{V_h}\Psi)\|_{0}+\|\Pi_{V_h}\Psi-\Psi\|_0+h^{\varepsilon}\|z\|_1)\|\Pi_{\ell}^o(\nabla\cdot\bm q)-\nabla\cdot\bm q\|_0 \\
            & \le Ch^{\varepsilon}(\|\Phi\|_{\varepsilon}+\|\Psi\|_1+\|z\|_1)\|\Pi_{\ell}^o(\nabla\cdot\bm q)-\nabla\cdot\bm q\|_0.
    \end{align*}
    From \eqref{estu-key}, $R_3$ and $R_5$ can be estimated as follows
    \begin{align*}
       & |R_3| +|R_5| \\ &\quad \le Ch^{\varepsilon}\|\Pi_{V_h}u-u\|_{0}(
      \|\bm\Phi\|_{\varepsilon}+
      \|\Psi\|_{1}
      )
      \\
       & \quad \quad
      +Ch\|\Pi_{V_h}u-u\|_0
      \|\nabla\cdot\bm \Phi\|_1
      +
      C\|\Pi_{V_h}u-u\|_{-1,D}(\omega^2\|\nabla\Psi\|_0+\|\nabla z\|_0).
    \end{align*}
    We find from the integration by parts and error estimate \eqref{err1} that
    \begin{align*}
      |R_4| & = \omega^2|(\nabla\cdot(\bm{\Pi}_{\bm P_h}\bm q-\bm q),\eta\Psi_h-\Pi_{\ell}^o(\eta\Psi_h))| \\
            & \le C\omega^2h\|\Pi_{\ell}^o(\nabla\cdot\bm q)-\nabla\cdot\bm q\|_0\|\nabla\Psi_h\|_0.
    \end{align*}
    From \eqref{neg_norm}, it is evident that
    \begin{align*}
      |R_6|\le C\omega^4\|\Pi_{V_h}u-u\|_{-1,D}\|\nabla\Psi\|_0.
    \end{align*}
    Combining all above results and Young's inequality with the estimate of \\$\|\Pi_{V_h}u-u_h\|_0$ in \eqref{l2-err-u-1}, we obtain that
      \begin{align*}
        \|\sigma^{\frac 1 2}\nabla(\Pi_{V_h}u-u_h)\|_0
         & \le C h^{\varepsilon} \left(
        \|\sigma^{-\frac 1 2}(\bm\Pi_{\bm P_h}\bm q-\bm q)\|_0
        +
        \|\Pi_{\ell}^o(\nabla\cdot\bm q)-\nabla\cdot\bm q\|_0\right)                            \\
         & \quad+C ( (h^{\varepsilon}+\omega^2) \|\Pi_{V_h}u-u\|_{0} +\|\Pi_{V_h}u-u\|_{-1,D}),
      \end{align*}
      and when $k\ge1$, together with the estimate of $\|\Pi_{V_h}u-u_h\|_0$ in \eqref{l2-err-u-2}, it holds
    \begin{align*}
       & \|\sigma^{\frac 1 2}\nabla(\Pi_{V_h}u-u_h)\|_0                                 \\
       & \quad \le C (h^{\varepsilon}+\omega^2 h^{\beta+\varepsilon-1}) \left(
      \|\sigma^{-\frac 1 2}(\bm\Pi_{\bm P_h}\bm q-\bm q)\|_0
      +
      \|\Pi_{\ell}^o(\nabla\cdot\bm q)-\nabla\cdot\bm q\|_0\right)                      \\
       & \quad\quad+C ( h^{\varepsilon} \|\Pi_{V_h}u-u\|_{0} +\|\Pi_{V_h}u-u\|_{-1,D}),
    \end{align*}
    which completes the proof.
  \end{proof}}

\begin{remark}[\zzya{Supercloseness} estimates]
  For $\Omega$ satisfies $\beta=2$ and  $\varepsilon=1$, when the elements $\bm{\mathcal{RT}}_{m-1}/\mathcal P_{m}$, $m>1$ are used, it holds from \eqref{the7.3_2} that
  \begin{align*}
    \|\sigma^{\frac 1 2}\nabla(\Pi_{V_h}u-u_h)\|_0\le Ch^{m+1}(\|q\|_m+\|\nabla\cdot q\|_{m}+\|u\|_{m+1}).
  \end{align*}
  For $\Omega$ satisfies $\beta=1$ and  $\varepsilon=1$, and for the element  $\bm{\mathcal{RT}}_{0}/\mathcal P_{1}$, by \eqref{the7.3_1}, the \zzya{supercloseness} estimate holds
  \begin{align*}
    \|\sigma^{\frac 1 2}\nabla(\Pi_{V_h}u-u_h)\|_0\le Ch^{2}(\|q\|_1+\|\nabla\cdot q\|_{1}+\|u\|_{2}).
  \end{align*}
\end{remark}

\zzy{Next, we state some results based on the regularity of the data and the properties of the domain $\Omega$.}

\zzy{
  \begin{remark}\label{rem_singdata}
    For the case $u\in H^{2+t}(\Omega),t> -1/2$, choosing $\bm{\mathcal{BDM}}_{k}, k \ge t$ or $\bm{\mathcal{RT}}_{k}, k \ge 1+t$ and $\mathcal P_{m},m\ge1+t$, we can conclude from \eqref{reg_con_2} that
    \begin{align*}
      \|\Pi_{V_h}u-u_h\|_0 + \|\nabla\cdot(\bm{\Pi}_{\bm P_h}\bm q-\bm q_h)\|_0 \le Ch^{2+t}(\|u\|_{2+t}+\|\bm q\|_{1+t}) & \leq Ch^{2+t}\| f\|_t, \\
      \|\bm{\Pi}_{{\bm P}_h}\bm q-\bm q_h\|_0 + \|\nabla(\Pi_{V_h}u-u_h)\|_0\le Ch^{1+t}(\|u\|_{2+t}+\|\bm q\|_{1+t})     & \leq Ch^{1+t}\| f\|_t.
    \end{align*}
  \end{remark}
}
\zzy{For the case $\Gamma_D=\partial\Omega$ and $\Omega$ is convex, $t$ can be in $(-1/2,0]$. Moreover, for $\Gamma_D=\partial\Omega$, $d=2$ and the largest angle of $\partial \Omega$ is strictly less than 90 degrees, $t$ can be in $(-1/2,1]$.}

\section{Convergence tests}\label{nume_test}
\zzy{In this section, we provide the experimental results for two categories of $f$, one is smooth and the other is singular.
}
\subsection{\zzy{Smooth data}}\label{sec_smooth}
\zzy{We} set $\Omega=(-1,1)^2$, $\sigma=\eta=1$, $\Gamma_D=\partial\Omega$, $\Gamma_N=\phi$, $\bm g=\bm 0$ and $f$ is chosen such that the exact solution is
\zzy{
  \begin{align*}
    u(x,y) = (x^2-1)(y^2-1)e^x.
  \end{align*}}
\zzy{The experiments were coded using MFEM and the system was solved using MINRES, with AMS and AMG employed as block preconditioners. We use an unstructured mesh.}

\zzy{Fig 1 correspond to $\bm{\mathcal{RT}}_{k}/\mathcal P_{m}$ with different pairs of $(k, m)$ while Fig 2 correspond to $\mathcal{BDM}_{k}/\mathcal P_{m}$. In these \zzyb{tests}, we choose $\omega = 1$.}

\zzy{It is observed that when we use elements $\bm{\mathcal{BDM}}_{4}/\mathcal P_{3},\bm{\mathcal{RT}}_{3}/\mathcal P_{3},\bm{\mathcal{RT}}_{4}/\mathcal P_{3}$, the convergence rate of $\|\Pi_{V_h}u -u_h\|_0$ attains $o(h^6)$, which exceeds the \zzya{supercloseness} rates established in our analyses. Moreover, using $\bm{\mathcal{RT}}_{0}/\mathcal P_{1}$, $\bm{\mathcal{RT}}_{1}/\mathcal P_{2}$, $\bm{\mathcal{RT}}_{2}/\mathcal P_{1}$, $\bm{\mathcal{BDM}}_{4}/\mathcal P_{3}$ yield a convergence rate for  $\|\bm{\Pi}_{\bm P_h}\bm q-\bm q_h\|_0$ that exceeds the theoretical expectations. Furthermore, when employing $\bm{\mathcal{RT}}_{2}/\mathcal P_{2}$, $\bm{\mathcal{RT}}_{3}/\mathcal P_{3}$, $\bm{\mathcal{RT}}_{3}/\mathcal P_{2}$, $\bm{\mathcal{RT}}_{4}/\mathcal P_{3}$, we observe a convergence rate for $\|\nabla(\Pi_{V_h}u -u_h)\|_0$ that surpasses the theoretical results.}

\zzy{The different convergence rates of $\|\nabla\cdot(\bm{\Pi}_{\bm P_h}\bm q-\bm q_h)\|_0$ in \Cref{table3} \zzyb{come} from the selections of $\omega$,} and the numerical result confirm our analysis in \Cref{rem_su_div_q}. In addition to the above situations, the numerical results match the remarks and estimates perfectly, and strongly confirm our theoretical analysis.

\begin{table}[H]\tiny
  \centering
  \scalebox{1.5}{%
    \begin{tabular}{c|c|c|c|c}
      \Xhline{1pt}
      \multirow{2}{*}{DOF} & \multicolumn{2}{c|}{$\omega = 0$} & \multicolumn{2}{c}{$\omega = 1$}                     \\
      \cline{2-5}
                           & Error                             & Rate                             & Error      & Rate \\
      \hline

      87                   & 6.1626e-03                        & -                                & 1.2534e-02 & -    \\
      313                  & 8.3053e-04                        & 2.89                             & 2.8069e-03 & 2.16 \\
      1185                 & 1.0684e-04                        & 2.96                             & 7.2319e-04 & 1.96 \\
      4609                 & 1.3478e-05                        & 2.99                             & 1.8327e-04 & 1.98 \\
      18177                & 1.6873e-06                        & 3.00                             & 4.5991e-05 & 1.99 \\
      72193                & 2.1082e-07                        & 3.00                             & 1.1509e-05 & 2.00 \\
      \Xhline{1pt}
    \end{tabular}%
  }
  \caption{Results of $\|\nabla\cdot(\bm{\Pi}_{\bm P_h}\bm q-\bm q_h)\|_0$ for $\bm{\mathcal{BDM}}_1/\mathcal{P}_2$ elements with different $\omega$}
  \label{table3}
\end{table}

\begin{figure}[H]
  \centerline{
    \hbox{\includegraphics[width=1\linewidth]{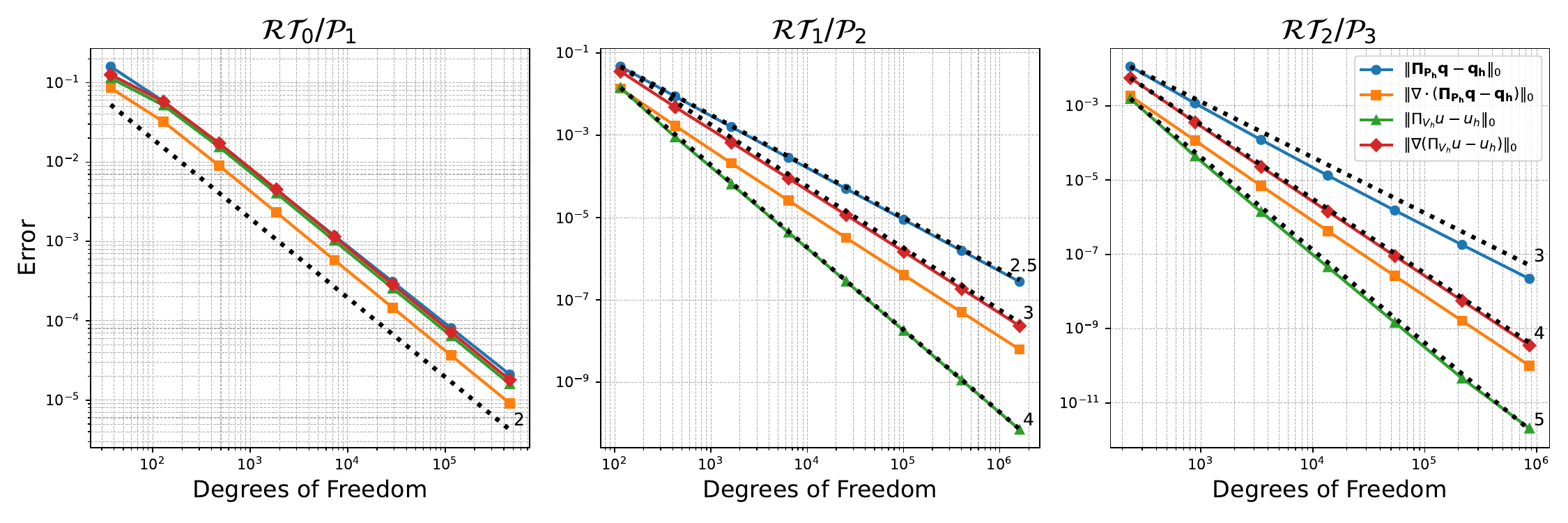}}}
  \centerline{
    \hbox{\includegraphics[width=1\linewidth]{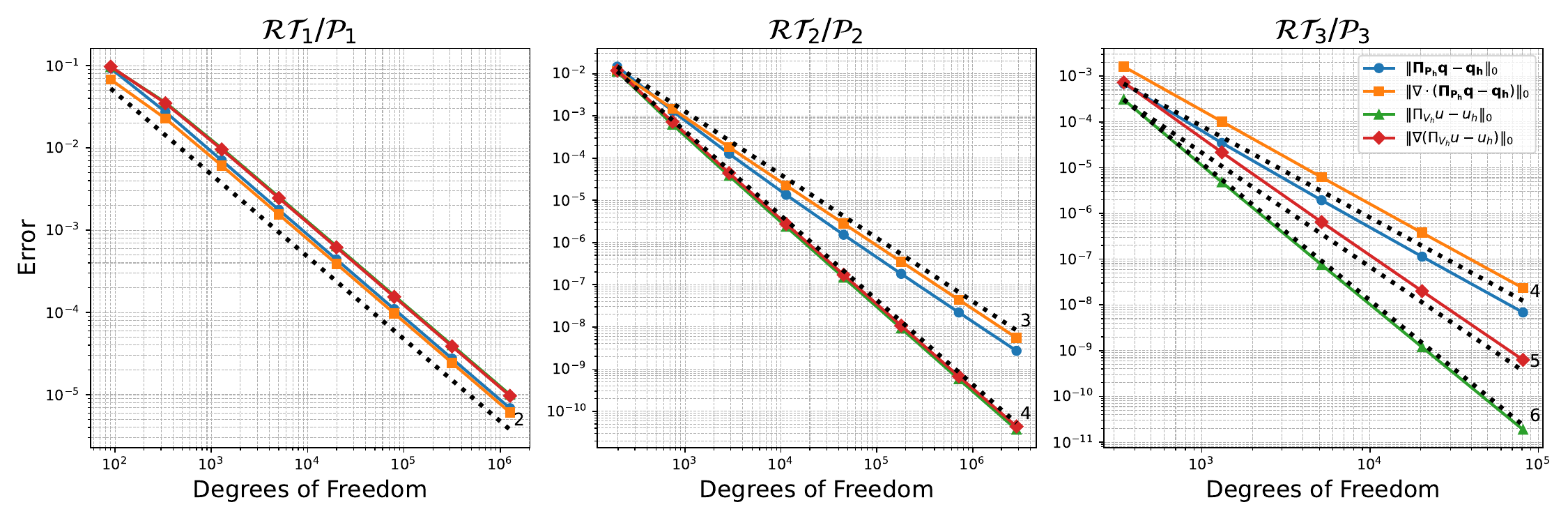}}}
  \centerline{
    \hbox{\includegraphics[width=1\linewidth]{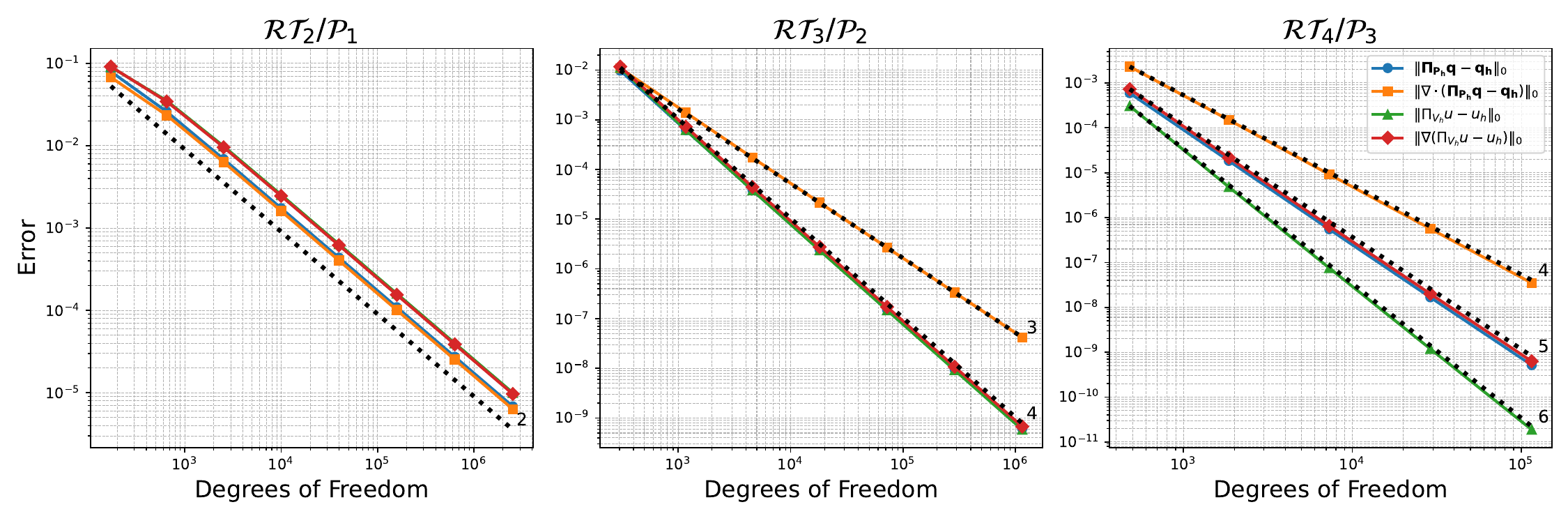}}}
  \caption{Results for $\bm{\mathcal{RT}}$ elements}
\end{figure}
\begin{figure}[H]
  \centerline{
    \hbox{\includegraphics[width=1\linewidth]{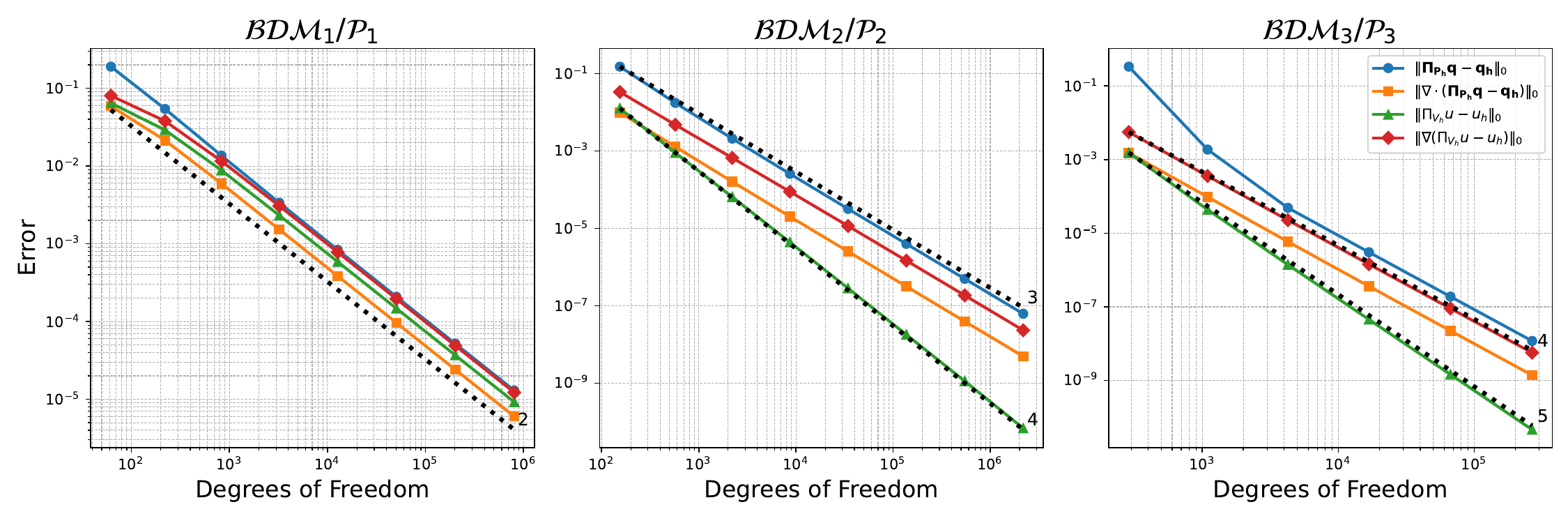}}}
  \centerline{
    \hbox{\includegraphics[width=1\linewidth]{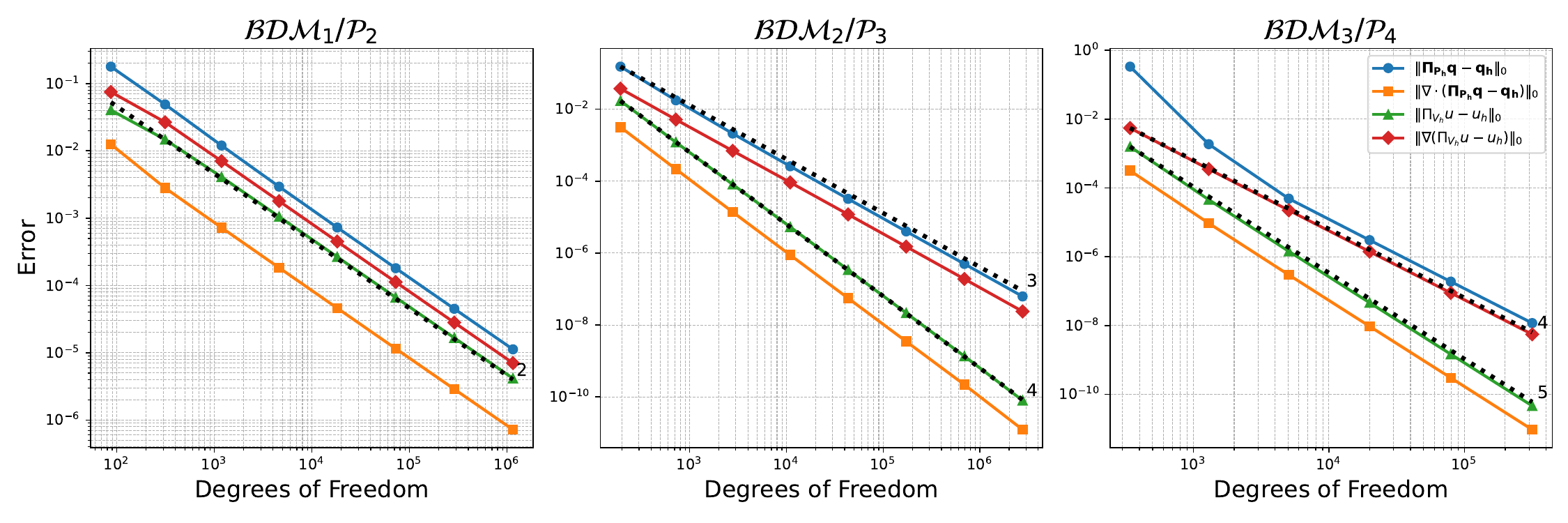}}}
  \centerline{
    \hbox{\includegraphics[width=1\linewidth]{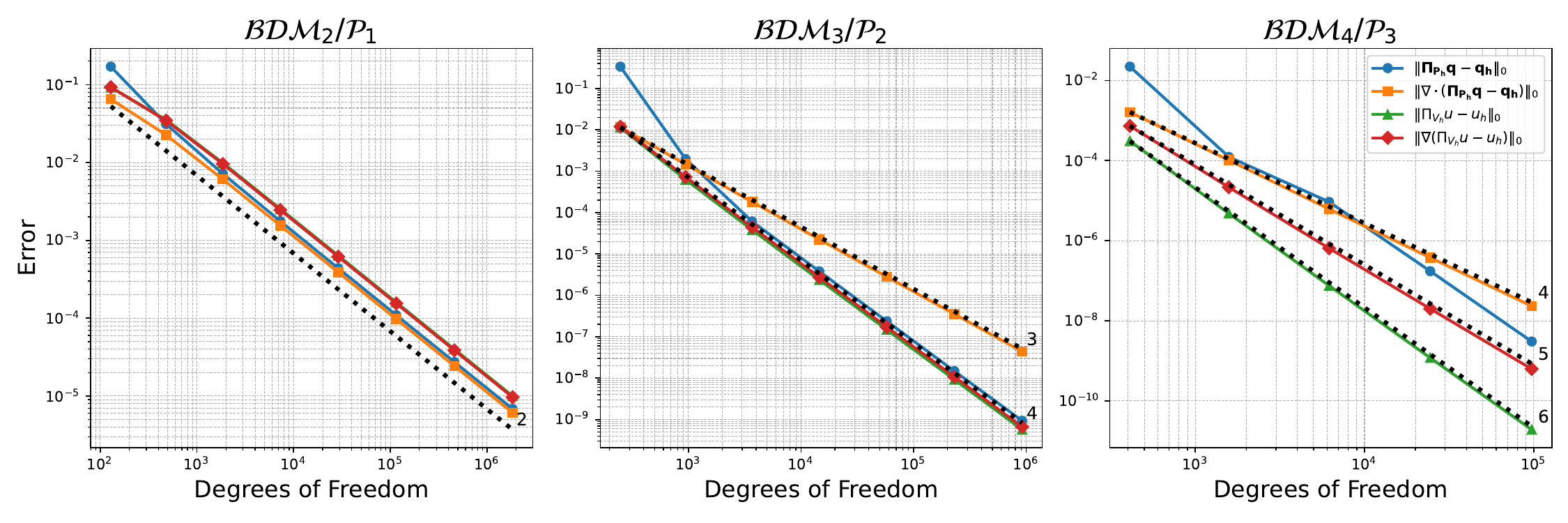}}}
  \caption{Results for $\bm{\mathcal{BDM}}$ elements}
\end{figure}

\zzya{Furthermore, we conducted experiments with different wavenumber $\omega$ and a functional form of $\eta$ and $\sigma$. We set $\omega = 2, 4, 6, 8$, $\eta(x,y) = (x^2-x)(y^2-y)$ and $\sigma(x,y) = x^2+y^2+1$. As shown in Fig 3, in the cases of $\omega = 2,4$, the numerical results are consistent with our analysis. In the cases of $\omega = 6,8$, the errors converger slowly at the coarse meshes. As we stated in \Cref{rem_hiwave}, the large wavenumber cases require a different scheme.}
\begin{figure}[H]
  \centerline{
    \hbox{\includegraphics[width=1\linewidth]{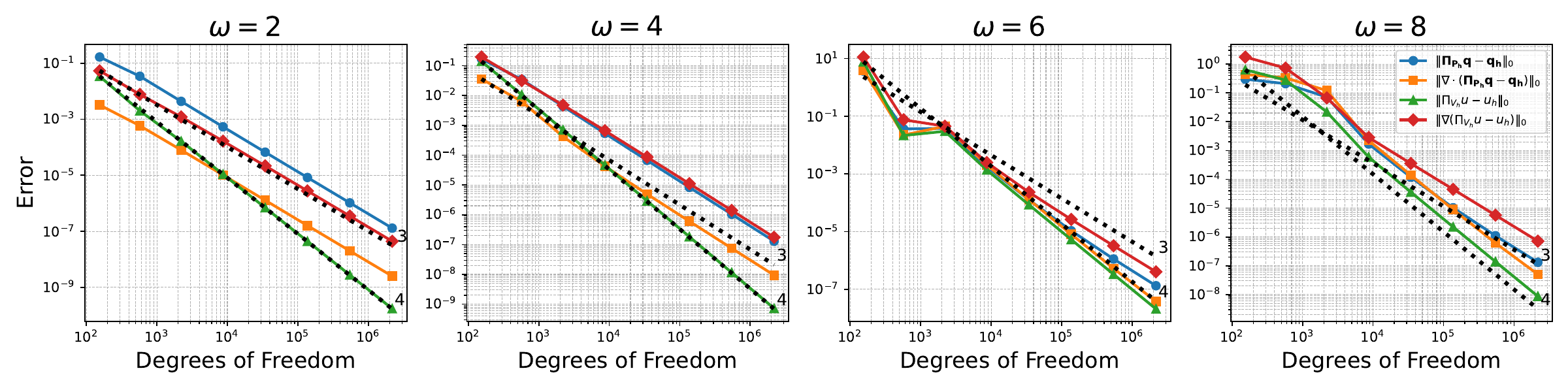}}}
  \caption{\zzya{Results for $\bm{\mathcal{BDM}}_2/\mathcal{P}_2$ elements}}
\end{figure}

\subsection{\zzy{Singular data}}
\zzy{For singular data, we set $\Omega=(-1,1)^2$, $\sigma=\eta=1$, $\omega = 0$, $\Gamma_D=\partial\Omega$, $\Gamma_N=\phi$, $\bm g=\bm 0$ and $u(x,y) = v(x)w(y)$, where
  \begin{align*}
     & v(x) = x|x|^{1/2+1/4}(1-x^2),\quad w(y)=(1-y^2),          \\
     & f(x,y) = \frac{sign(x)(165x^2-21)}{16x^{1/4}}w(y)+ 2v(x).
  \end{align*}
  Note that $u\notin H^{2+t}(\Omega), \bm q = \sigma \nabla u\notin \bm H^{1+t}(\Omega),\nabla \cdot\bm q\notin H^{t}(\Omega)$ for $t\ge 1/4$. For the convenience of dealing with singular integrals, we rely on a structured mesh. The numerical results are presented in \zzyb{Fig 4 and Fig 5}. \zzyb{Moreover, the theoretical convergence rates are shown in Talbles 5 and 6.}}

\begin{figure}[H]
  \centerline{
    \hbox{\includegraphics[width=1\linewidth]{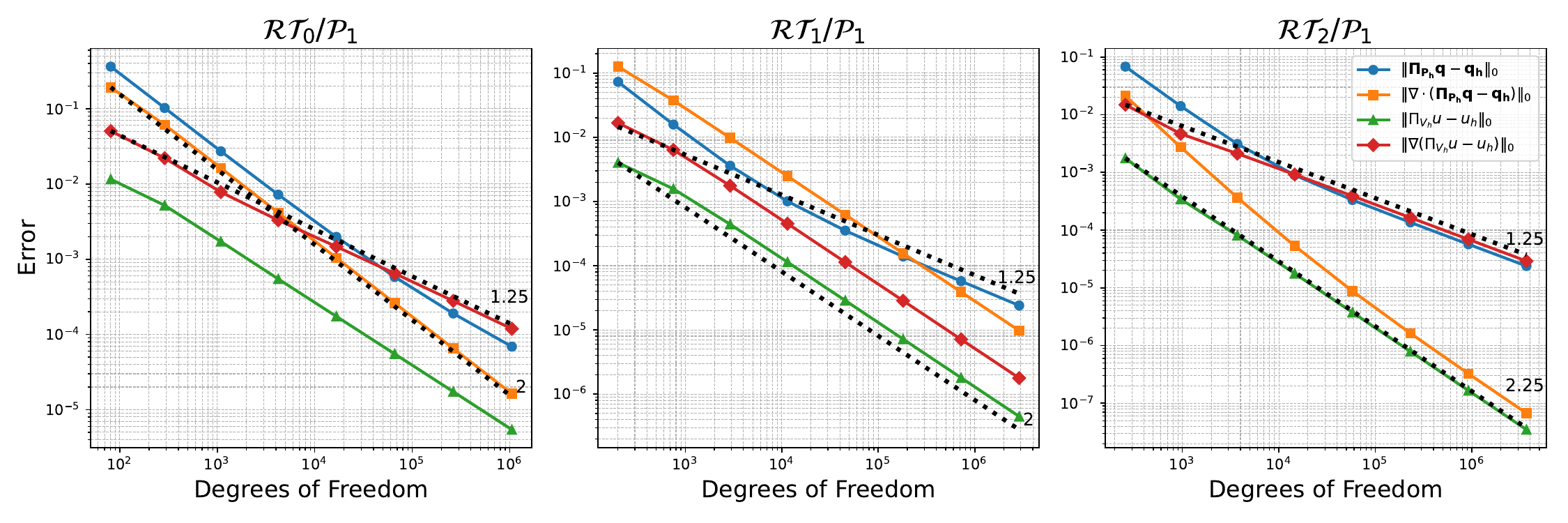}}}
  \caption{Results for $\bm{\mathcal{RT}}$ elements}
\end{figure}
\begin{table}[H]
  \centering
  \scalebox{0.9}{
    \begin{tabular}{c|c|c|c|c|c}
      \Xhline{1pt}
      {$\bm P_h$}                  & {$V_h$}        & {$\|\bm{\Pi}_{\bm P_h}\bm q-\bm q_h\|_0$} & {$\|\nabla\cdot(\bm{\Pi}_{\bm P_h}\bm q-\bm q_h)\|_0$} & {$\|\Pi_{V_h}u-u_h\|_0$} & {$\|\nabla(\Pi_{V_h}u-u_h)\|_0$} \\
      \hline

      $\bm{\bm{\mathcal{RT}}}_{0}$ & $\mathcal P_1$ & (1.25)                                    & 2.00                                                   & 1.25                     & 1.25                             \\
      \hline
      $\bm{\bm{\mathcal{RT}}}_{1}$ & $\mathcal P_1$ & 1.25                                      & 2.00                                                   & 2.00                     & (1.25)                           \\
      \hline
      $\bm{\bm{\mathcal{RT}}}_{1}$ & $\mathcal P_2$ & 1.25                                      & 2.25                                                   & 2.25                     & 1.25                             \\

      \Xhline{1pt}
    \end{tabular}}
  \caption{Theoretical convergence rates. The numbers in $(\cdot)$ indicate that the convergence rates obtained from numerical experiments are better than the theoretical predictions. The convergence rates observed in numerical experiments align with the theoretical rates in \Cref{rem_singdata} when employing $\bm{\bm{\mathcal{RT}}}_{1}/\mathcal P_2$.}
\end{table}

\begin{figure}[H]
  \centerline{
    \hbox{\includegraphics[width=1\linewidth]{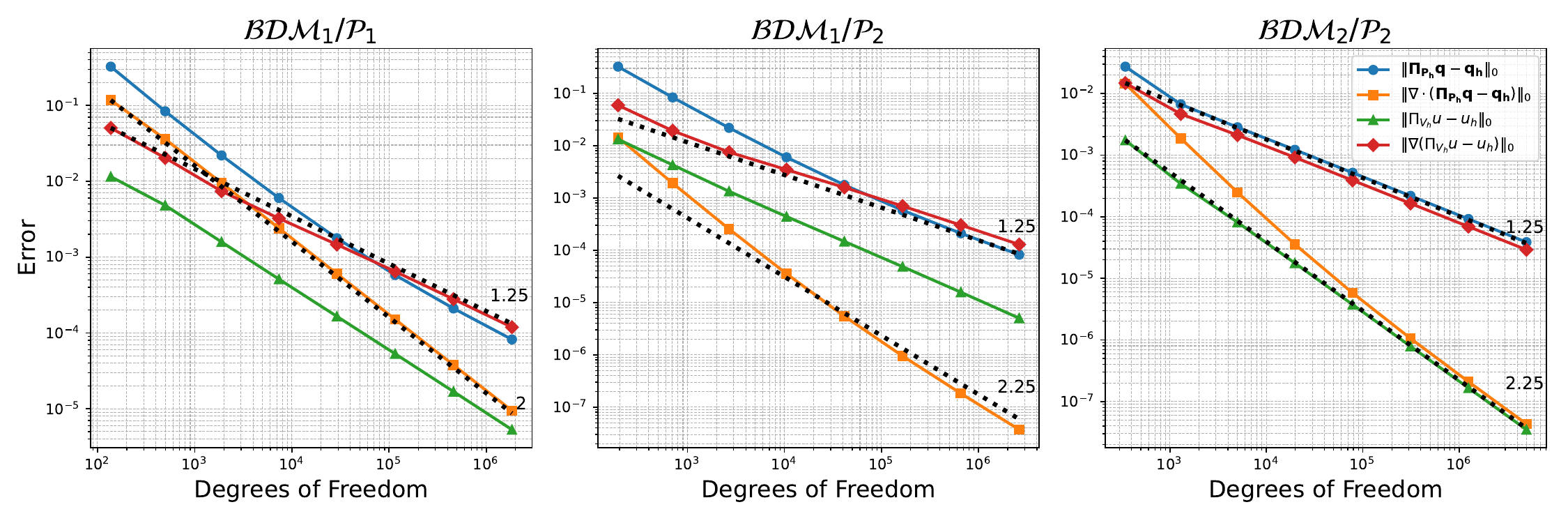}}}
  \caption{Results for $\bm{\mathcal{BDM}}$ elements}
\end{figure}
\begin{table}[H]
  \centering
  \scalebox{0.9}{
    \begin{tabular}{c|c|c|c|c|c}
      \Xhline{1pt}
      {$\bm P_h$}              & {$V_h$}        & {$\|\bm{\Pi}_{\bm P_h}\bm q-\bm q_h\|_0$} & {$\|\nabla\cdot(\bm{\Pi}_{\bm P_h}\bm q-\bm q_h)\|_0$} & {$\|\Pi_{V_h}u-u_h\|_0$} & {$\|\nabla(\Pi_{V_h}u-u_h)\|_0$} \\
      \hline

      $\bm{\mathcal{BDM}}_{1}$ & $\mathcal P_1$ & (1.25)                                    & 2.00                                                   & (1.25)                   & 1.25                             \\
      \hline
      $\bm{\mathcal{BDM}}_{1}$ & $\mathcal P_2$ & 1.25                                      & 2.25                                                   & 1.25                     & 1.25                             \\
      \hline
      $\bm{\mathcal{BDM}}_{2}$ & $\mathcal P_2$ & 1.25                                      & 2.25                                                   & 2.25                     & 1.25                             \\
      \Xhline{1pt}
    \end{tabular}}
  \caption{Theoretical convergence rates. The numbers in $(\cdot)$ indicate that the convergence rates obtained from numerical experiments are better than the theoretical predictions.}
\end{table}

\subsection{\zzya{Superconvergence}}
\zzya{We try to apply postprocessing techniques to obtain superconvergence error estimates based on supercloseness results. For example, the superconvergence of $\|\nabla(u^*_h - u)\|_0$ can be derived using the following element-by-element postprocessing \cite{Cockbur2010}: Find $(u^*_h, \lambda_h)\in \mathcal P_{m+1}(T)\times \mathcal P_{0}(T)$ such that for all $(v_h, \zeta_h)\in \mathcal P_{m+1}(T)\times \mathcal P_{0}(T)$
  \begin{align*}
    (\sigma \nabla u^*_h, \nabla v_h)_T + (\lambda_h, v_h)_T & = (\bm q_h, \nabla v_h)_T, \\
    (u^*_h, \zeta_h)_T                                       & =(u_h, \zeta_h)_T,
  \end{align*}
  where $(\bm q_h, u_h)\in\bm{\mathcal{BDM}}_{k}(\text{ or }\bm{\mathcal{RT}}_{k})/\mathcal P_m$ is the solution of \eqref{lsfem}.}
  
  \zzya{We performed experiments using the same exact solution as in \Cref{sec_smooth}, and set $\omega = 1$, $\eta(x,y) = (x^2-x)(y^2-y)$ and $\sigma(x,y) = x^2+y^2+1$. As shown in Fig 6, when using elements $\bm{\mathcal{BDM}}_{k}/\mathcal P_k, \bm{\mathcal{RT}}_{k}/\mathcal P_k, k = 1,2,3$, $\|\nabla(u^*_h - u)\|_0$ exhibits superconvergence. To obtain superconvergence results of $\|\nabla\cdot(\bm q^*_h -\bm q)\|_0$ and $\|u^*_h - u\|_0$, we need to design more subtle postprocessing. This topic will be investigated in future research.
}
\begin{figure}[H]\label{BMD_post}
  \centerline{
    \hbox{\includegraphics[width=0.8\linewidth]{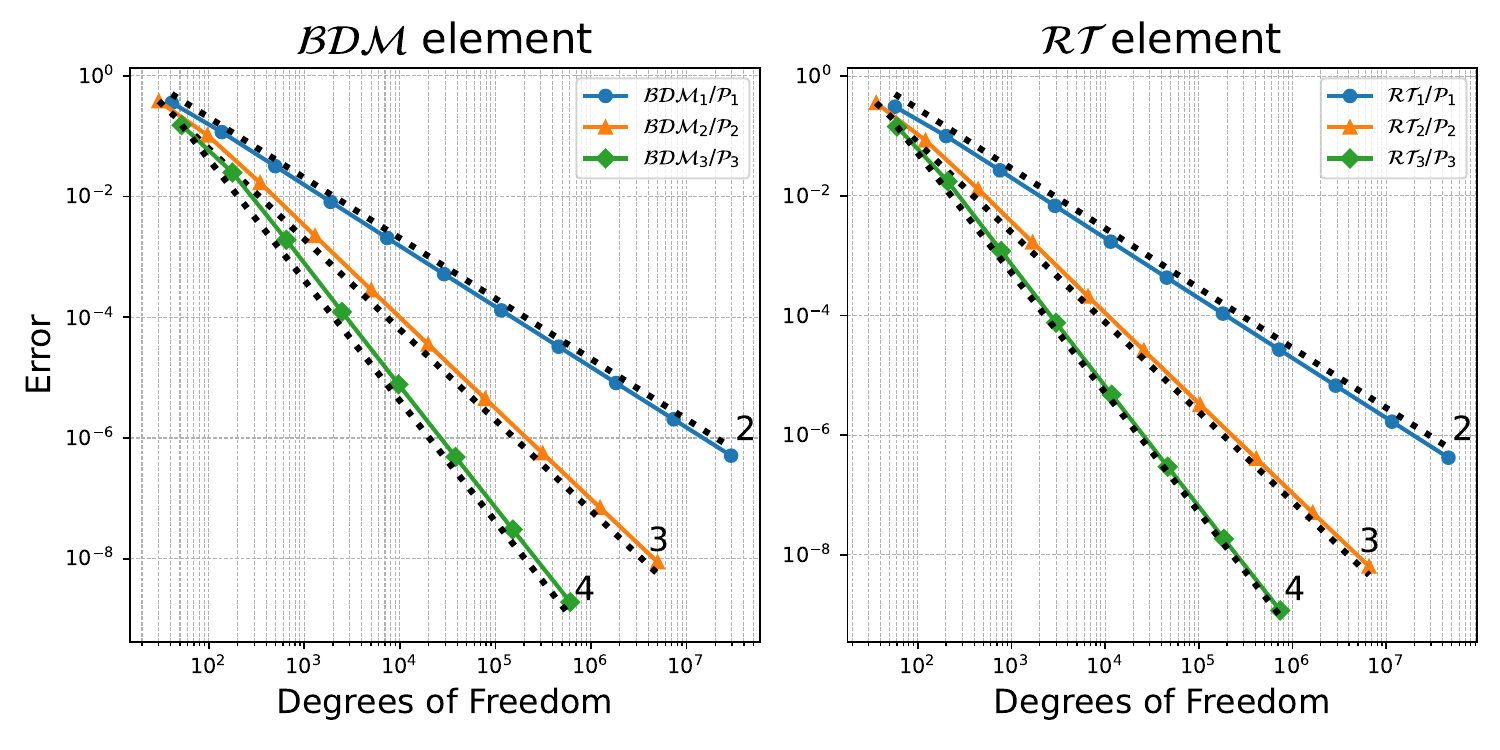}}}
  \caption{\zzya{Superconvergence from postprocessing}}
\end{figure}

\section*{Declarations}

\subsection*{Funding}
Gang Chen and Zheyuan Zhang are supported by National Natural Science Foundation of China (NSFC) under grant 121713413 and 12422115, and Opening Foundation of Agile and Intelligent Computing Key Laboratory of Sichuan Province.
Fanyi Yang is supported by \zzyb{Natural Science Foundation of Sichuan under grant 2023NSFSC1323.}

\subsection*{Conflict of Interest}
The authors declare that they have no conflict of interest.

\subsection*{Data availability}
Data sharing not applicable to this article as no datasets were generated or analysed during the current study.

\bibliographystyle{spmpsci}
\bibliography{reference}
\end{document}